\documentclass[10pt]{amsart} 
\usepackage{mathrsfs}
\usepackage{amssymb,amsmath,amsthm,color}
\usepackage{graphicx}
\usepackage{url}
\definecolor{forestgreen(traditional)}{rgb}{0.0, 0.27, 0.13}
\definecolor{forestgreen(web)}{rgb}{0.13, 0.55, 0.13}
\definecolor{green(html/cssgreen)}{rgb}{0.0, 0.5, 0.0}
\usepackage[colorlinks=true,linkcolor=blue,citecolor=green]{hyperref}
\usepackage{verbatim}
\usepackage{setspace}
\usepackage{mathtools}
\usepackage{multirow}
\usepackage{array}
\usepackage{stackengine}
\usepackage{tikz}
 \usepackage[pagewise]{lineno}
\usepackage[inline]{enumitem}  
\usepackage[margin=1in]{geometry}
\usepackage{comment}
\numberwithin{equation}{section}
\newtheorem{theorem}{Theorem}[section]
\newtheorem{dfn}[theorem]{Definition}
\newtheorem{prp}[theorem]{Proposition}
\newtheorem{lemma}[theorem]{Lemma}
\newtheorem{Corollary}[theorem]{Corollary}
\newtheorem{rmk}[theorem]{Remark}
\newcommand{\R}{\mathbb R}
\newcommand{\rd}{\mathbb{R}^d}
\newcommand{\n}{\Vert}
\newcommand{\C}{\mathbb C}
\newcommand{\K}{\widetilde{K}}

\newcommand{\la}{\mathcal{L}_a}
\newcommand{\hd}{\dot{H}^1(\rd)}
\newcommand{\ha}{\dot{H}_a^1(\rd)}
\newcommand{\F}{\mathcal{F}}
\newcolumntype{P}[1]{>{\centering\arraybackslash}p{#1}}
\newcolumntype{M}[1]{>{\centering\arraybackslash}m{#1}}
\newcommand{\abs}[1]{\left\vert#1\right\vert}
\begin{document}
\baselineskip16pt
\title[Inhomogeneous Strichartz Estimates]{Strichartz Estimates for  Schr{\"o}dinger equation with singular and time dependent Potentials and Application to NLS equations}

\author{Saikatul Haque}
\address{TIFR - Centre for Applicable Mathematics\\
Bengaluru-560065, India}
\email{saikatul@tifrbng.res.in}

\maketitle
\begin{abstract}  
We establish inhomogeneous Strichartz Estimates for the Schr{\"o}dinger equation with singular and time-dependent potentials for some non-admissible pairs. Our work extends the results of Vilela \cite{vilela2007inhomogeneous} and Foschi \cite{foschi2005inhomogeneous}, where they proved the results in the absence of potential. It also extends the works of Pierfelice \cite{pierfelice2006strichartz} and Burq, Planchon, Stalker, Tahvildar-Zadeh \cite{burq2003strichartz}, who proved the estimates for admissible pairs. We also extend the recent work of Mizutani, Zhang, Zheng \cite{mizutani2020uniform}, and as an application of it, we improve the stability result of Kenig-Merle \cite{kenig2006global}, which in turn establishes a proof (alternative to \cite{yang2020scattering0}) of the existence of scattering solution for the energy-critical focusing NLS with inverse-square potentials. 
\end{abstract}


\section{Introduction}
Let us consider the following Cauchy problem for the Schr{\"o}dinger equation
\begin{equation}\label{cp}
i\partial_t u+\Delta u+Vu=F\ \text{in }\R\times\rd,\quad
u(0,\cdot)=0\ \text{on }\rd
\end{equation} 
where $u:\R\times\rd\to\C$ is the unknown, $V:\rd\times\R\to\R$ is a real-valued potential and $F:\R\times\rd\to\C$ 
is a given function. 
This equation plays an important role in quantum mechanics and has been studied extensively  when $V=0$, see \cite{strichartz1977restrictions,yajima1987existence,cazenave1988cauchy,keel1998endpoint,ginibre1985global}. 
In this case, since the operator $\Delta$ is self-adjoint in $H^s(\rd)$ for all $s\in\R$, 
by semi-group theory, the existence of the unique solution $e^{it\Delta}f$ of the corresponding {\it homogeneous problem}
\begin{equation}\label{cp0}
i\partial_t v+\Delta v=0\ \text{in }\R\times\rd,\quad
v(0,\cdot)=f\ \text{on }\rd
\end{equation} 
 is ensured for $f\in H^s(\rd)$ for all $s\in\R$, in particular for $f\in L^2(\rd)$. 
Note that, applying Fourier transform $\F$ to \eqref{cp0}, 
and solving the resulting ODE for $\F v$ one finds $\F v(t)=e^{-4\pi^2it|\xi|^2}\F f$. Then taking 
inverse Fourier transform, it follows that $ e^{it\Delta}f=v(t)$ is given by $e^{it\Delta}f=M_tD_t\mathcal{F}M_tf$ for $t\neq0$, where $M_t w=e^{i|\cdot|^2/4t}w$, $D_tw={(4\pi it)^{-d/2}}w \left({\cdot}/{4\pi t}\right)$, see \cite[Remark 2.2.5]{cazenave2003semilinear}. This formula suggests that the operators $e^{it\Delta}$, $t\neq0$ has a lot of similarities with the Fourier transform operator $\F$. In fact it turns out that, $e^{it\Delta}f$ satisfies the $L^\infty$-$L^1$ estimates, called the {\it dispersive estimate} 
\begin{equation}\label{dp}
\left\n e^{it\Delta}f\right\n_{L^\infty}\lesssim t^{-d/2}\n f\n_{L^1},\quad t\neq0,
\end{equation} which can be seen as a variant 
of the estimate $\n\F f\n_{L^\infty}\lesssim\n f\n_{L^1}$. Using the dispersive estimate \eqref{dp}, the {\it inhomogeneous Strichartz estimate}
 \begin{equation}\label{str}
\left\n u\right\n_{L^qL^r}\lesssim\n F\n_{L^{\tilde{q}'}L^{\tilde{r}'}}
\end{equation}
 is established for admissible pairs $(q,r),(\tilde{q},\tilde{r})$ with $q,\tilde{q}\neq2$. From the special case $(q,r)=(\tilde{q},\tilde{r})$, using a duality argument the {\it homogeneous Strichartz estimate}
  \begin{equation}\label{str0}
\left\n e^{it\Delta}f\right\n_{L^qL^r}\lesssim\n f\n_{L^2}
\end{equation} is derived for admissible pair  $(q,r)$ with $q\neq2$, see \cite[Chapter 4]{linares2014introduction} for details. A standard scaling argument shows that
\begin{equation}
\label{adm}
\frac{2}{q}+\frac{d}{r}=\frac{d}{2}
\end{equation}  is necessary for the validity of the estimate \eqref{str0}. Recall that 
 a pair of exponent $(q,r)$ is called an {\it admissible pair} if $q,r\geq2$, $(q,r,d)\neq(2,\infty,2)$ and the necessary condition \eqref{adm} is satisfied. 

The Strichartz estimates have many applications in solving non-linear Schr{\"o}dinger (NLS) equations with various kinds of nonlinearities, e.g. power-type nonlinearity (i.e. $F=\pm|u|^\alpha u$), Hartree-type non-linearity (i.e. $F=\pm(|x|^{-\gamma}\ast|u|^2)u$).      Motivated by these non-linear problems (see also e.g. \eqref{61}, in Section \ref{s4}), our focus in this work is to establish Strichartz estimates for solutions to \eqref{cp} involving some wide class of space-time spaces.

 The inequalities \eqref{str} and \eqref{str0} go back to 1977, when Strichartz  \cite{strichartz1977restrictions} proved the special case $q=\tilde{q}=r=\tilde{r}={2(d+2)}/{d}$ as a Fourier restriction Theorem. Later Ginibre-Velo \cite{ginibre1985global} in 1985, 
Yajima \cite{yajima1987existence} in 1987 and Cazenave, Weissler \cite{cazenave1988cauchy} in 1988 proved \eqref{str}, \eqref{str0} assuming $(q,r)$, $(\tilde{q},\tilde{r})$ are admissible pairs and $q\neq2,\tilde{q}\neq2$.  The remaining case for admissible pairs $(q,r)$, $(\tilde{q},\tilde{r})$ i.e. when at least one of $q,\tilde{q}$ is $2$ (the endpoint case)  is due to Keel and Tao  \cite{keel1998endpoint} where they proved \eqref{str}, \eqref{str0} for more general settings in the year 1998.

Let us now concentrate on the {\it inhomogeneous estimate} i.e. \eqref{str}. 
Again by a rescaling argument, $q,\tilde{q},r,\tilde{r}$ must satisfy
\begin{equation}\label{sc}
\frac{2}{q}+\frac{2}{\tilde{q}}+d\left(\frac{1}{r}+\frac{1}{\tilde{r}}\right)=d
\end{equation}
whenever \eqref{str} holds. Note that the relation \eqref{sc} is satisfied for many choices of $q,\tilde{q},r,\tilde{r}$ apart from those for which $(q,r),(\tilde{q},\tilde{r})$ are admissible pairs.  This indicates the possibility of \eqref{str} being true for non-admissible pairs. 
In fact, for non-admissible pairs, various authors including Cazenave, Weissler \cite{cazenave1992rapidly} in 1992, Kato \cite{kato1994q} in 1994, Foschi \cite{foschi2005inhomogeneous} in 2005, Vilela \cite{vilela2007inhomogeneous} in 2007, Koh \cite{koh2011improved} in 2011, proved the inequality \eqref{str} for $q,r,\tilde{q},\tilde{r}$ satisfying \eqref{sc} and other restrictions. But the problem of finding all possible exponents satisfying the estimate \eqref{str}, is still open.

Now one \textbf{question} arises: what happens  when $V$ is non-zero? Assume that $(q,r),(\tilde{q},\tilde{r})$ are admissible pairs. 
 It follows from \cite[Theorem 1.2]{keel1998endpoint} that, for any self-adjoint operator $H$ in $L^2(\rd)$ satisfying the dispersive estimate
 \begin{equation}\label{dp0}
\left\n e^{itH}f\right\n_{L^\infty}\lesssim t^{-d/2}\n f\n_{L^1},\quad t\neq0,
\end{equation} 
 the Strichartz estimates \begin{equation}\label{str1}
 \left\n e^{itH}f\right\n_{L^qL^r}\lesssim\n f\n_{L^2}
 \end{equation} and \eqref{str} hold for $u$ satisfying 
 \begin{equation}\label{cp1}
 i\partial_t u+Hu=F\ \text{in }\R\times\rd,\quad u(0,\cdot)=0\ \text{on }\rd.
 \end{equation}
 Therefore, to have the Strichartz estimates \eqref{str1}, \eqref{str} for a solution to \eqref{cp1}, it is sufficient to have inequality \eqref{dp0}. Let us consider 
the case when $H$ has the particular form 
$H=\Delta+V$ where $V:\rd\to\R$ is a given function. This case is intensively studied, 
e.g. if the positive part of $V$ is not too large, then it has been shown that $H$ is self-adjoint, see Kato \cite{kato2013perturbation}. Schonbek \cite{schonbek1979decay} showed if $\n V\n_{L^1\cap L^\infty}$ is sufficiently small then the estimate \eqref{dp0} holds. It was also proved that  if $V\in C^\infty(\rd)$ is non-positive and $D^\alpha V\in L^\infty(\rd)$ for all $\alpha\geq2$ then \eqref{dp0} holds, see e.g. Fujiwara \cite{fujiwara1979construction}, Weinstein \cite{weinstein1985symbol}, Zelditch \cite{zelditch1983reconstruction} and Oh \cite{oh1988existence}.

In this work, we consider potentials $V$ in  $L^{d/2,\infty}(\rd)$ (and in $L^\infty(\R,L^{d/2,\infty}(\rd))$). Note that these $V$'s need not fall in the previous categories and hence the validity of the dispersive estimate \eqref{dp0} is not ensured.  
Therefore we possibly need a different kind of machinery to deal with such potentials.  For time-independent $V$, the operator $\Delta+V$ is self-adjoint  in $L^2(\rd)$ (via Friedrich’s extension)  
for the cases: 
\begin{itemize}
\item[(i)] $V$ is of the form $a/|x|^2$ with $a<(d-2)^2/4$ for $d\geq3$, by Hardy inequality, (or see \cite[Introduction]{killip2017energy}),
\item[(ii)] $V$ with sufficiently small $\n V\n_{L^{d/2,\infty}}$, see \cite[Section 2]{pierfelice2006strichartz}.
\end{itemize}
The case (i) above was studied by Burq--Planchon--Stalker--Tahvidar-Zadeh \cite{burq2003strichartz} in 2003.  Using spherical harmonics and Hankel transforms the authors established the estimates \eqref{str} and  \eqref{str1}  for admissible pairs $(q,r),(\tilde{q},\tilde{r})$.
On the other hand, the case (ii) above was considered by Pierfelice \cite{pierfelice2006strichartz} in 2006 to prove the inhomogeneous estimate\begin{align}
\left\n u\right\n_{L^qL^{r,2}}&\lesssim\n F\n_{L^{\tilde{q}'}L^{\tilde{r}',2}} \label{37}
\end{align} for admissible pairs $(q,r),(\tilde{q},\tilde{r})$.
Note that by Calder\'on's result i.e. Lemma \ref{cld} below, it follows that, \eqref{37} is stronger than \eqref{str} for $2\leq r,\tilde{r}<\infty$, see Corollary \ref{cld1} below.
The author in \cite{pierfelice2006strichartz} also presented proof of the existence of a solution to \eqref{cp} for time-dependent potentials via fixed point argument. A similar problem was studied by Bouclet and Mizutani \cite{bouclet2018uniform} in 2018, where the authors provided estimates, for potentials in Morrey-Campanato spaces. 

Here we would like to ask another \textbf{question}: what happens when the exponents  $q,r,\tilde{q},\tilde{r}$ are such that $(q,r),(\tilde{q},\tilde{r})$ are not admissible pairs? 
The only result according to our knowledge, answering the above two questions is the very recent (in 2020) work of Mizutani, Zhang, Zheng \cite{mizutani2020uniform}, where they proved the inhomogeneous Strichartz estimate \eqref{37}, for some non-admissible pairs in the case (i) above, see Theorem \ref{56} (ii) below. 
The case (ii) above is completely open for non-admissible pairs according to our knowledge. In this article, we establish the inhomogeneous estimate \eqref{37} for $V$ satisfying the case (ii) above, with appropriate exponents $1\leq q,\tilde{q},r,\tilde{r}\leq\infty$ for which $(q,r),(\tilde{q},\tilde{r})$ need not be admissible, see Theorems \ref{1}, \ref{1'}. 

 To achieve these 
 estimates, first, we improve the result of Vilela \cite[Theorem 2.4]{vilela2007inhomogeneous}. We would like to point out that the author in \cite{vilela2007inhomogeneous} proved the estimate \eqref{str} in the zero potential case, whereas we in the following result establish the stronger estimate \eqref{37}:
 
\begin{theorem}\label{3'}
Let $V=0$ and $(q,r),(\tilde{q},\tilde{r})$ satisfy \eqref{sc}, $r,\tilde{r}>2$ along with
\begin{equation}\label{2}
\frac{d-2}{d}<\frac{r}{\tilde{r}}<\frac{d}{d-2},\quad 
\textcolor{black}{q,\tilde{q}\geq2,\quad}\quad\\
\begin{cases}
\frac{d}{2}\left(\frac{1}{r}-\frac{1}{\tilde{r}}\right)<\frac{1}{\tilde{q}}\quad\text{if }r\leq\tilde{r}\\
\frac{d}{2}\left(\frac{1}{\tilde{r}}-\frac{1}{r}\right)<\frac{1}{q}\quad\text{if }\tilde{r}\leq r.
\end{cases}
\end{equation} Then the inhomogeneous Strichartz estimate \eqref{37} holds for a solution $u$ to \eqref{cp}.
\end{theorem}

Because of the scaling condition \eqref{sc}, once we fix $r,\tilde{r},q$, the exponent $\tilde{q}$ is determined. Theorem \ref{3'} tells that the estimate \eqref{37} with $V=0$, holds on the pentagon $ACDEF$, for some $q$'s, see Figure \ref{f1}. To prove Theorem \ref{3'} we crucially use the Lemma \ref{11} below due to Vilela \cite[Lemma 2.2]{vilela2007inhomogeneous}. 
 It is worth noticing here that the author in \cite{vilela2007inhomogeneous} also presented some negative result, namely, it was shown that  if \eqref{sc} is satisfied and $1/r+1/\tilde{r}<(d-2)/d$ or if $1/r,1/\tilde{r}$ is outside the pentagon $ACD'E'F$ in Figure \ref{f1}, then the estimate \eqref{str} (and hence \eqref{37}) does not hold for $V=0$.

As mentioned earlier, The above result is used to go from zero potential to non-zero potential case by using perturbation technique, incorporated from \cite{pierfelice2006strichartz}, followed by interpolations for mixed Lebesgue/Lorentz spaces. 
By $2^*,p_*$ we mean the standard Sobolev conjugate $2d/(d-2)$ of $2$ and the number $p(d-1)/(d-2)$ (for $d\geq3$) respectively. Set $2_*^*=(2^*)_*$ and note that $2<2_*<2^*<2_*^*$. Here is our main result, answering the questions asked earlier: 

\begin{theorem}\label{1}
Let $d\geq3$ and $(q,r),(\tilde{q},\tilde{r})$
satisfy \eqref{sc}, \eqref{2} and\textcolor{black}{
\begin{equation}\label{c1}
\begin{cases}
2_*<r<2_*^*,\  \textit{(the region BCDE)}\\
d\left|\frac{1}{r}-\frac{1}{2^*}\right|<\frac{1}{2}=\frac{1}{q}
\end{cases}or\ \ 
\begin{cases}
2_*<\tilde{r}<2_*^*,\ \textit{(the region DEFG)}\\
d\left|\frac{1}{\tilde{r}}-\frac{1}{2^*}\right|<\frac{1}{2}=\frac{1}{\tilde{q}}
\end{cases}
\end{equation}}
Let $V$ be a real-valued potential with $c_s(\frac{dr}{2r+d})'\n V\n_{L^{d/2,\infty}}({\it or\ }c_s(\frac{d\tilde{r}}{2\tilde{r}+d})'\n V\n_{L^{d/2,\infty}}{\it  respectively})<1$
(here $c_s$ is the constant appearing in the Strichartz estimates for the unperturbed equation). For the region $DEFG$ i.e. the second set of conditions in \eqref{c1}, we further assume, $\n V\n_{L^{d/2,\infty}}$ is so small that, $\Delta+V$ is self-adjoint. Then the inhomogeneous Strichartz estimate \eqref{37} holds for a solution $u$ to \eqref{cp}. \\
Moreover, a similar result holds for time-dependent potential in the region $BCDE$, if $V$ satisfies the smallness condition $c_s(\frac{dr}{2r+d})'\n V\n_{L^\infty L^{d/2,\infty}}<1$. 
\end{theorem}

\begin{rmk}{\rm
 Our results Theorems \ref{1} and \ref{1'} below extend the results of Pierfelice  \cite[Theorems 1, 3]{pierfelice2006strichartz}, Cazenave, Weissler \cite{cazenave1992rapidly}, Kato \cite{kato1994q}, and Vilela \cite[Theorem 2.4]{vilela2007inhomogeneous}. }
\end{rmk}

By symmetry of the problem we conclude that if the estimate \eqref{37} is true for $(q,r)=(q_0,r_0),(\tilde{q},\tilde{r})=(\tilde{q}_0,\tilde{r}_0)$ then the estimate \eqref{37} is also true for $(q,r)=(\tilde{q}_0,\tilde{r}_0),(\tilde{q},\tilde{r})=(q_0,r_0)$ provided $\Delta+V$ is self-adjoint. In that case, the plot of $(1/r,1/\tilde{r})$, for which the estimate \eqref{37} holds for some $q$, becomes symmetric with respect to the line $1/r=1/\tilde{r}$, in $1/r$ verses $1/\tilde{r}$ coordinate, see Figure \ref{f1}. 
Therefore it is enough to prove this result for the first set of conditions in \eqref{c1}, as we impose further smallness of $\n V\n_{L^{d/2,\infty}}$ for the region $DEFG$ so that $\Delta+V$ becomes self-adjoint. 
\begin{rmk}\label{rm1}
{\rm
Interpolating the results in the two regions mentioned in \eqref{c1}, we would derive that the estimate \eqref{37} also holds in the triangular region $BGH$.} 
 \end{rmk} 
Now interpolating Theorem \ref{1} (with \textcolor{black}{$2_*<r<2_*^*$}) and the result of Pierfelice \cite[Theorems 1, 3]{pierfelice2006strichartz}  we conclude the following:

\begin{theorem}\label{1'}
Let $d\geq3$ and $V$ be a real-valued potential with $c_s2^*\n V\n_{L^{d/2,\infty}}<1$. 
Then the inhomogeneous Strichartz estimate \eqref{37} holds for a solution $u$ to \eqref{cp} provided
 $(q,r),(\tilde{q},\tilde{r})$
satisfy the scaling condition \eqref{sc}, \textcolor{black}{
$
\frac{2q(d-1)}{q(d-1)-2}<r<\frac{2qd(d-1)}{qd(d-1)-6d+8},\left|1-\frac{q}{2d}(\frac{1}{2}-\frac{1}{r})\right|<\frac{1}{2},\tilde{r}>2
$ and
\begin{equation}
\begin{cases}
\frac{1}{2}-\frac{1}{\tilde{r}}-\frac{1}{q}<\frac{d}{2}(\frac{1}{r}-\frac{1}{\tilde{r}})<\frac{1}{r}-\frac{1}{2}+\frac{1}{q},\ 2\leq q\leq\tilde{q}\\
\frac{d}{2}(\frac{1}{r}-\frac{1}{\tilde{r}})<\frac{1}{\tilde{q}}\quad\text{if }r\leq\tilde{r}\\
\frac{d}{2}(\frac{1}{\tilde{r}}-\frac{1}{r})<\frac{1}{q}\quad\text{if }\tilde{r}\leq r.
\end{cases}or\ \ 
\begin{cases}
\frac{d-1}{d}(1-\frac{2}{q})-\frac{1}{\tilde{r}}<\frac{d}{2}(\frac{1}{{r}}-\frac{1}{\tilde{r}})<\frac{1}{r}, \ \tilde{q}\geq2\\
\frac{d}{2}(\frac{1}{{r}}-\frac{1}{\tilde{r}})<\frac{1}{\tilde{q}} \quad\text{if }\frac{1}{r}-\frac{1}{\tilde{r}}\geq\frac{1}{d}(1-\frac{2}{q})\\
\frac{d}{2}(\frac{1}{\tilde{r}}-\frac{1}{r})<\frac{1}{q}\quad\text{if }\frac{1}{r}-\frac{1}{\tilde{r}}\leq\frac{1}{d}(1-\frac{2}{q})
\end{cases}
\end{equation}}
Moreover, a similar result holds for time-dependent potential, if $c_s2^*\n V\n_{L^\infty L^{d/2,\infty}}<1$. 
\end{theorem}

From Theorems \ref{1},  \ref{1'} and Remark \ref{rm1}, we conclude that if $\n V\n_{L^\infty L^{d/2,\infty}}$ is small enough, then the estimate \eqref{37} holds for $1/r,1/\tilde{r}$ in the pentagonal region $ACDEF$ (in Figure \ref{f1}, the line $DE$ is excluded), with some $q$'s, for which the pairs $(q,r),(\tilde{q},\tilde{r})$ need not be admissible.

\begin{center}
\begin{figure}\label{f1}
\begin{tikzpicture}
\draw[->] (0,0)--(5,0) node[anchor=west] {\tiny{$\frac{1}{r}$}};
\draw[->] (0,0)--(0,5) node[anchor=north east] {\tiny{$\frac{1}{\tilde{r}}$}};
\draw (0,0) rectangle (4,4);
\fill[gray!40!white](2,2)--(2/3,2)--(1/3,1)--(1,1/3)--(2,2/3)--(2,2);
\draw[dashed,gray] (0,2)--(4,2)node[anchor=west]{\textcolor{black}{\tiny{${1}/{2}$}}};
\draw[dashed,gray] (2,0)--(2,4)node[anchor=south]{\textcolor{black}{\tiny{$\frac{1}{2}$}}};
\draw[dashed,gray] (0,0)--(4,4/3)node[anchor=west]{\textcolor{black}{\tiny{${2}/{2^*}$}}};
\draw[dashed,gray] (0,0)--(4/3,4)node[anchor=south]{\textcolor{black}{\tiny{$\frac{2}{2^*}$}}};
\draw[dashed,gray] (0,4/3)--(8/3,4)node[anchor=south]{\textcolor{black}{\tiny{$\frac{d+1}{2^*}$}}};
\draw[dashed,gray] (4/3,0)--(4,8/3)node[anchor=west]{\textcolor{black}{\tiny{${(d+1)}/{2^*}$}}};
\draw[dashed,gray] (4/3,0)--(0,4/3);
\draw[dashed,gray] (1/3,0)--(1/3,4)node[anchor=south]{\textcolor{black}{\tiny{$\frac{1}{2_*^*}$}}};
\draw[dashed,gray] (2/3,0)--(2/3,4)node[anchor=south]{\textcolor{black}{\tiny{$\frac{1}{2^*}$}}};
\draw[dashed,gray] (1,0)--(1,4)node[anchor=south]{\textcolor{black}{\tiny{$\frac{1}{2_*}$}}};
\draw[dashed,gray] (0,1/3)--(4,1/3)node[anchor=west]{\textcolor{black}{\tiny{${1}/{2_*^*}$}}};
\draw[dashed,gray] (0,2/3)--(4,2/3)node[anchor=west]{\textcolor{black}{\tiny{${1}/{2^*}$}}};
\draw[dashed,gray] (0,1)--(4,1)node[anchor=west]{\textcolor{black}{\tiny{${1}/{2_*}$}}};
\draw[dashed,gray] (2,1)--(1,2);
\filldraw[black](4,0)circle(1pt)node[anchor=north ]{\tiny{$(1,0)$}};
\filldraw[black](0,4)circle(1pt)node[anchor=east ]{\tiny{$(0,1)$}};
\filldraw[black](2,2)circle(0pt)node[anchor=south west ]{\tiny{$A$}};
\filldraw[black](1.15,1.9)circle(0pt)node[anchor=south]{\tiny{$B$}};
\filldraw[black](.55,1.9)circle(0pt)node[anchor=south]{\tiny{$C$}};
\filldraw[black](.4,.85)circle(0pt)node[anchor=east]{\tiny{$D$}};
\filldraw[black](.85,.4)circle(0pt)node[anchor=north]{\tiny{$E$}};
\filldraw[black](2,2/3)circle(0pt)node[anchor=north west]{\tiny{$F$}};
\filldraw[black](1.9,1.15)circle(0pt)node[anchor=west]{\tiny{$G$}};
\filldraw[black](0,4/3)circle(0pt)node[anchor=east]{\tiny{$D'$}};
\filldraw[black](4/3,0)circle(0pt)node[anchor=north]{\tiny{$E'$}};
\filldraw[black](.9,.9)circle(0pt)node[anchor=south west]{\tiny{$H$}};
\end{tikzpicture}
\caption{\small{Strichartz estimate \eqref{37} holds in the shaded region (side $DE$ excluded) for certain values of $q$ for $\n V\n_{L^\infty L^{d/2,\infty}}$ sufficiently small. Note that it is only accurate for the case $d=3$, for larger $d$ the points $D'(\frac{d-3}{2d},\frac{d-1}{2d})$, $E'(\frac{d-1}{2d}\frac{d-3}{2d})$ might be inside the first quadrant, and the line $D'E'$ might cut the lines $AC$ and $AF.$}}
\end{figure}
\end{center}

Next, we state the inhomogeneous estimates for inverse-square potentials. Note that the first two results are from \cite{burq2003strichartz} and \cite{mizutani2020uniform} and we derive the third case as a generalization of the first two cases. 

\begin{theorem}\label{56}
Let $d\geq3$, $a\in(-\infty,\frac{(d-2)^2}{4})$, $V=\frac{a}{|\cdot|^2}$ and $0<\gamma,\tilde{\gamma}\leq1$. 
Then the  Strichartz estimate
\begin{enumerate}
\item[(i)] \eqref{str}, \eqref{str1} holds for admissible pairs $(q,r),(\tilde{q},\tilde{r})$,
\item[(ii)] \eqref{37} holds for $q=\tilde{q}=2,r=\frac{2d}{d-2s},\tilde{r}=\frac{2d}{d-2(2-s)}$ provided $s\in A_{a,1}$, 
\item[(iii)] \eqref{str} holds for $q=\frac{2}{\gamma},\tilde{q}=\frac{2}{\tilde{\gamma}},r=\frac{2d}{d-2(s+\gamma-1)},\tilde{r}=\frac{2d}{d-2(1+\tilde{\gamma}-s)}$ provided $s\in A_{a,\gamma\tilde{\gamma}}$, 
\end{enumerate}
where $A_{a,\lambda}=\left(1-\frac{d-2}{2(d-1)}\lambda,1+\frac{d-2}{2(d-1)}\lambda\right)\cap R_{a,\lambda}$, ($0<\lambda\leq1$) and $R_{a,\lambda}$ is given by\begin{equation*}
R_{a,\lambda}=\begin{cases}
\hspace{2.40cm}\left(1-\frac{\lambda}{2},1+\frac{\lambda}{2}\right),\hspace{2.85cm}\text{if }\ \quad\sqrt{\frac{(d-2)^2}{4}-a}>\frac{1}{2}\\
\left(1-\frac{(d-2)^2-4a}{2(2+4a-(d-2)^2)}\lambda,1+\frac{(d-2)^2-4a}{2(2+4a-(d-2)^2)}\lambda\right),\quad\text{if } 0<\sqrt{\frac{(d-2)^2}{4}-a}\leq\frac{1}{2}.
\end{cases}
\end{equation*}
\end{theorem}
Part (iii) of the above result can be rewritten in relatively less complicated way as follows: \eqref{str} holds  provided $q,\tilde{q},r,\tilde{r}$ satisfy \eqref{sc}, $q,\tilde{q}\geq2$ and $\big|d\big(\frac{1}{2}-\frac{1}{r}\big)-\frac{2}{q}\big|<\frac{2}{q\tilde{q}}\min\big\{\frac{d-2}{d-1},\frac{(d-2)^2-4a}{2+4a-(d-2)^2}\big\}$.
Now we compare the Theorems \ref{1}, \ref{1'} with Theorem \ref{56}: 
As inverse square potentials belong to $L^{d/2,\infty}(\rd)$, Theorems \ref{1} and \ref{1'}  are applicable to potentials of the form $a/\abs{\cdot}^2$, with $\abs{a}$ sufficiently small. 
Since 
 $\{a/\abs{\cdot}^2:a\in\R\}\subsetneq L^{d/2,\infty}(\rd)$,  
 Theorems \ref{1} and \ref{1'} cover larger collection of potentials than Theorem \ref{56}\footnote{ the original version of 
Theorem \ref{56} (ii) i.e. \cite[Theorem 1.3]{mizutani2020uniform} 
covers more general potential $V$ of the form $V(x)=v(\theta)r^{-2}$ where $r=|x|,\theta=x/|x|$, $v\in C^1(\mathbb{S}^{d-1})$ but this is also a proper subclass of $L^{d/2,\infty}(\rd)$.} when $|a|$ is sufficiently small. Another advantage of Theorems \ref{1}, \ref{1'} is they accommodate time-dependent potentials which Theorem \ref{56} does not cover. 
 Note that we can take $(1/r,1/\tilde{r})$ close to $(1/2^*,1/2)$  in the shaded region in figure \ref{f1} so that estimate \eqref{37} holds (for certain choices of $q$) using Theorem \ref{1'}, but such choice of $r,\tilde{r}$ is not applicable to Theorem \ref{56}\footnote{This can be proved by noting that  $A_{a,\gamma\tilde{\gamma}}=\big(1-\frac{d-2}{2(d-1)}\gamma\tilde{\gamma},1+\frac{d-2}{2(d-1)}\gamma\tilde{\gamma}\big)$ for sufficiently small $|a|$.  For this choice of $r,\tilde{r}$, one has $\gamma\tilde{\gamma}\sim(s-1)(2-s)$.  But $s\not\in A_{a,(s-1)(2-s)}\sim A_{a,\gamma\tilde{\gamma}}$.
}. 
On the other hand, there are exponents which are in acceptable range of Theorem \ref{56} but not covered by Theorems \ref{1}, \ref{1'} (e.g. one can choose $r=\infty$ in Theorem \ref{56} but it is not applicable to  Theorems \ref{1}, \ref{1'}).
A big advantage of Theorem \ref{56} is that, it covers all $a/\abs{\cdot}^2$ with $a<(d-2)^2/4$ even if $\abs{a}$ is large.
 Thus Theorem \ref{1'} (together with  Theorem \ref{1})  and Theorem \ref{56} have edge over one another in the acceptable range of various parameters.

As an application of Theorem \ref{56} (iii), we obtain a  {\bf Long time perturbation} result with inverse square potential (see Theorem \ref{ltp}), improving the result by Kenig-Merle \cite[Theorem 2.14]{kenig2006global}. This, in turn, gives a proof (an alternative to \cite{yang2020scattering0}) of the scattering result for focusing energy-critical NLS with inverse-square potential, see Theorem \ref{mt}.

We organize the material as follows: In Section \ref{s2} the notations and some known results are mentioned, in Section \ref{s3} we present the proofs of results.  In the end, in Section \ref{s4}, we provide the Long time perturbation result and its application to NLS.

\section{Preliminaries}\label{s2}
\subsection{Notations}
Throughout this article we denote by $\n \cdot \n$ and $\langle\cdot,\cdot\rangle$ the $L^2(\rd)$ norm and inner product respectively unless otherwise specified.\\
By $l_q^\beta$, we denote the weighted sequence space $L^q(\mathbb{Z},2^{j\beta}dj)$, where $dj$ stands for counting measure.\\
The Lorentz space is the space of all complex-valued measurable functions $f$ such that $\n f\n_{L^{r,s}(\rd)}<\infty$ where $\n f\n_{L^{r,s}(\rd)}$ is defined by
\begin{equation}\label{ls}
\n f\n_{L^{r,s}(\rd)}:=r^{\frac{1}{s}}\left\n t\mu\{|f|>t\}^{\frac{1}{r}}\right\n_{L^s\left((0,\infty),\frac{dt}{t}\right)}
\end{equation} with $0<r<\infty$, $0<s\leq\infty$ and $\mu$ denotes the Lebesgue measure on $\rd$. Therefore 
\begin{equation*}
\n f\n_{L^{r,s}(\rd)}=\begin{cases}
r^{1/s}\left(\int_0^\infty t^{s-1}\mu\{|f|>t\}^{\frac{s}{r}}dt\right)^{1/s}\quad \text{for }s<\infty\\
\sup_{t>0}t\mu\{|f|>t\}^{\frac{1}{r}}\quad\quad\quad \text{for }s=\infty.
\end{cases}
\end{equation*}\\
For an interval $I\subset \mathbb R$  the norm of the space-time Lebesgue space $L^q(I, L^r(\mathbb R^d))$ is defined by
$\|u\|_{L^q (I,L^r(\rd))}:=\left(\int_{I} \|u(t)\|^q_{L^r} dt \right)^{1/q}.$ Similarly $L^q(I, L^{r,s}(\mathbb R^d))$ is defined. We write $\|u\|_{L^q(I,L^r)}$ for $\|u\|_{L^q (I,L^r(\rd))}$ and $\|u\|_{L^qL^r}$ for $\|u\|_{L^q(\R,L^r(\rd))}$. 
By $\n\cdot\n_{S(I)},\n\cdot\n_{W(I)}$ we denote
\begin{equation*}\label{SW}
\|u\|_{S(I)}\ = \ \|u\| _{L^{\frac{2(d+2)}{d-2}}\big(I, L^{\frac{2(d+2)}{d-2}}(\rd)\big) }\ , \ \|u\|_{W(I)} \ = \ \|u\| _{L^{\frac{2(d+2)}{d-2}}\big(I, L^{\frac{2d(d+2)}{d^2+4}}(\rd)\big) }.
\end{equation*}\\
The Fourier transform $\F f$ or $\widehat{f}$ of a function $f\in L^1(\rd)$ is defined by $\widehat{f}(\xi)= \int_{\rd} e^{-2\pi i x\cdot \xi}f(x)dx$. Note that it has a unique extension as an operator in $\mathcal{S}'(\rd)$ with the property that for each $u\in\mathcal{S}'(\rd)$, $\langle\F u,\varphi\rangle=\langle u,\F\varphi\rangle$ for all $\varphi\in \mathcal{S}(\rd)$.\\
The Sobolev conjugate ${2d}/{(d-2)}$ of $2$ is denoted by $2^*$. We set $p_*={p(d-1)}/{(d-2)}$ (for $d\geq3$) and $2_*^*=(2^*)_*$. By $a\vee b$ we mean $\max\{a,b\}$ and $a\lesssim b$ stands for $a\leq cb$ for some (universal) constant $c$.

\subsection{Sobolev spaces adapted with inverse-square potentials }
We denote the Homogeneous Sobolev spaces by $\dot{W}^{s,p}(\rd)$ which is defined as the completion of $C_c^\infty(\rd)$
with the norm$$\n u\n _{\dot{W}^{s,p}} = \n  (- \Delta) ^{s/2}u\n_{L^p}  ,\ \ u\in C_c^\infty(\rd).$$
where $\widehat{(- \Delta) ^{s/2}u }(\xi) = (2 \pi)^s |\xi|^s \widehat{u} (\xi) $. We will denote
$\dot{H}^s(\rd)\ = \ \dot{W}^{s,2}(\rd).$

Consider the operator $\mathcal{L}_a:= \Delta + {a}/{|x|^2}$  (defined on $L^2$ by standard Friedrichs extension) which is self-adjoint in $L^2(\rd)$ for $a<\left(\frac{d-2}{2}\right)^2$. Let us denote the unitary group of operators $\{e^{it\la}\}_{t\in\R}$ in $L^2(\rd)$ by $\{S_a(t)\}_{t\in\R}$.

 We define as before:
$\dot{W}^{s,p}_a(\rd)\ $= the completion of $  C_c^\infty(\rd)$ with $\n\cdot\n _{\dot{W}_a^{s,p}}$,
where
$$\n u\n _{\dot{W}_a^{s,p}} = \n  (- \la) ^{s/2}u\n_{L^p}  ,\ \ u\in C_c^\infty(\rd)$$
and 
$\dot{H}_a^s(\rd)\ = \ \dot{W}_a^{s,2}(\rd).$ We often use $\n\cdot\n_{\dot{H}^1},\n\cdot\n_{\dot{H}_a^1}$ for $\n\cdot\n_{\hd},\n\cdot\n_{\ha}$ respectively.

The fractional derivative $D^\alpha u$ (or $|\nabla|^\alpha u$) of $u\in\mathcal{S}'(\rd)$ is defined as $\widehat{(- \Delta) ^{\alpha/2}u }$ i.e. $\F D^\alpha u=(2\pi)^\alpha|\xi|^\alpha\widehat{u}$ for $\alpha\in\R$.

Note that 
$ \n u\n _{\dot{H}^1}^2 = \langle (- \Delta) ^{1/2}u, (- \Delta) ^{1/2}u\rangle=\langle - \Delta u, u\rangle= \n\nabla u \n^2$ 
and 
$ \n u\n _{\dot{H}_a^1}^2 = \langle (- \la) ^{1/2}u, (- \la) ^{1/2}u\rangle \ = \ \langle - \la u, u\rangle \ = \ \n\nabla u \n^2 \ -\ a\left\n{u}/{|x|}\right\n^2.$ Therefore 
 using the
 Hardy's inequality 
\begin{equation*}\label{hardy}
\int_{\rd}|\nabla u(x)|^2dx\ \ge \ (\frac{d-2}{2})^2\int_{\rd}\frac{|u(x)|^2}{|x|^2}dx\ , \ u\in C_c^\infty(\rd)
\end{equation*} we have the following result:
\begin{lemma}
The homogeneous spaces $\hd$ and $\ha$ are the same when $a<(d-2)^2/4$.
\end{lemma} 

\begin{lemma}[See Theorem 1.2 in \cite{killip2018sobolev}]\label{NormEq} Let $d\geq 3$ and $a<\left(\frac{d-2}{2}\right)^2$. Then the norms $\n \cdot\n _{\dot{W}^{s,p}}$ and $ \n \cdot\n _{\dot{W}^{s,p}_a}$ are equivalent and hence $\dot{W}^{s,p}_a(\rd)\ = \  \dot{W}^{s,p}(\rd) $ under any of the following two conditions :
\begin{enumerate}
\item[(i)] $p=2,\ -1\le s\le 1$.
\item[(ii)]$\frac{d}{d- q_a}< p< \frac{d}{s+ q_a}$ , $0<s<2$ where $q_a= \frac{d-2}{2}- \sqrt{(\frac{d-2}{2})^2-a}$.
\end{enumerate}
\end{lemma}

%

\begin{Corollary}\label{st_a}
Let $d\geq 3$ and $(q,r),(\tilde{q},\tilde{r})$ are  admissible pairs satisfying $q>\frac{2}{\sqrt{\left(\frac{d-2}{2}\right)^2-a}}$, then 
\begin{itemize}
\item[(i)] $\|\nabla S_a(t)f\|_{L^qL^r} \lesssim \|f\|_{\dot{H}^1},$
\item[(ii)] $\left\| \nabla \int_0^t S_a(t-s)F(s,x) ds \right\|_{L^{q}L^r}\lesssim \|\nabla F\|_{L^{\tilde{q}'}L^{\tilde{r}'}}.$
\end{itemize}
\end{Corollary}
\begin{proof}[{\bf Proof}]
Follows from Theorem \ref{56} and Lemma \ref{NormEq}.
\end{proof}
\subsection{Interpolation spaces }
Here we recall some results on interpolation spaces. For details on this subject, one can see the book \cite{bergh2012interpolation} of Bergh and L\"ofstr\"om. We define the real interpolation space $(A_0,A_1)_{\theta,\rho}$ ($0<\theta<1,1\leq\rho\leq\infty$) of two Banach spaces $A_0,A_1$ via the norm
\[\n u\n_{(A_0,A_1)_{\theta,\rho}}=\left(\int_0^\infty(t^{-\theta} K(t,u))^\rho\frac{dt}{t}\right)^{1/\rho},\quad K(t,u)=\inf_{u=u_0+u_1}\n u_0\n_{A_0}+t\n u_1\n_{A_1}
\]where the infimum is taken over $(u_0,u_1)\in A_0\times A_1$ such that $u=u_0+u_1$.

\begin{lemma}[Theorem 3.4 in \cite{o1963convolution}]\label{lh}
Let $\frac{1}{r}=\frac{1}{r_0}+\frac{1}{r_1}<1$ and $s\geq1$ is such that $\frac{1}{s}\leq\frac{1}{s_0}+\frac{1}{s_1}$. Then $f\in L^{r_0,s_0}(\rd)$ and $g\in L^{r_1,s_1}(\rd)$ imply $fg\in L^{r,s}(\rd)$ and $\n fg\n_{L^{r,s}}\leq r'\n f\n_{L^{r_0,s_0}}\n g\n_{L^{r_1,s_1}}$.
\end{lemma}
\begin{lemma}[See e.g. \cite{bergh2012interpolation,huang2019function}]\label{int}
Let $q_j,r_j,\tilde{q}_j,\tilde{r}_j\in[1,\infty]$, $j=0,1$. Let $q,r$ is such that $\frac{1}{q}=\frac{1-\theta}{q_0}+\frac{\theta}{q_1},$  $\frac{1}{r}=\frac{1-\theta}{r_0}+\frac{\theta}{r_1},$ $\frac{1}{\tilde{q}}=\frac{1-\theta}{\tilde{q}_0}+\frac{\theta}{\tilde{q}_1},$ $\frac{1}{\tilde{r}}=\frac{1-\theta}{\tilde{r}_0}+\frac{\theta}{\tilde{r}_1}$ for some $\theta\in[0,1]$. Then for $\mathcal{T}$ linear,
\begin{enumerate}
\item[(i)] $\mathcal{T}:L^{q_0}L^{r_0}\to L^{\tilde{q}_0}L^{\tilde{r}_0}$ and $\mathcal{T}:L^{q_1}L^{r_1}\to L^{\tilde{q}_1}L^{\tilde{r}_1}$ imply $\mathcal{T}:L^qL^r\to L^{\tilde{q}}L^{\tilde{r}}$. 
\item[(ii)] $\mathcal{T}:L^{q_0}L^{r_0,2}\to L^{\tilde{q}_0}L^{\tilde{r}_0,2}$ and $\mathcal{T}:L^{q_1}L^{r_1,2}\to L^{\tilde{q}_1}L^{\tilde{r}_1,2}$ imply $\mathcal{T}:L^qL^{r,2}\to L^{\tilde{q}}L^{\tilde{r},2}$.
\end{enumerate}
\end{lemma}
\begin{lemma}[Section 3.13, exercise 5(b) in \cite{bergh2012interpolation} and Theorems 1, 2 in  \cite{janson1988interpolation}]\label{18}
Let $A_0,A_1,B_0,B_1,C_0,C_1$ are Banach spaces and $\mathcal{T}$ be a bilinear operator such that
\begin{align*}\mathcal{T}:\begin{cases}
A_0\times B_0\longrightarrow C_0,\\
A_0\times B_1\longrightarrow C_1,\\
A_1\times B_0\longrightarrow C_1,
\end{cases}
\end{align*}
then whenever $0<\theta_0,\theta_1<\theta=\theta_0+\theta_1<1$, $1\leq p,q,r\leq\infty$ and $1\leq\frac{1}{p}+\frac{1}{q}$, we have
$$\mathcal{T}:(A_0,A_1)_{\theta_0,pr}\times (B_0,B_1)_{\theta_1,qr}\longrightarrow(C_0,C_1)_{\theta,r}.$$
\end{lemma}

\begin{lemma}[Theorems 5.2.1 and 5.6.1 in \cite{bergh2012interpolation}]\label{19}We have the following interpolation results:
\begin{enumerate}
\item[(i)] Let $r_0<r<r_1$ and $0<\theta<1$ be such that $\frac{1}{r}=\frac{1-\theta}{r_0}+\frac{\theta}{r_1}$, then for $r_0<p$ we have
$(L^{r_0},L^{r_1})_{\theta,p}=L^{r,p}$
\item[(ii)] Let $\beta_0<\beta_1$ and $0<\theta<1$ be such that $(1-\theta)\beta_0+\theta\beta_1=\beta$, then $(l_\infty^{\beta_0},l_\infty^{\beta_1})_{\theta,1}=l_1^\beta$.
\end{enumerate}
\end{lemma}

\begin{lemma}[Calder\'on, see e.g. Lemma 2.5 in \cite{o1963convolution}]\label{cld}
Let $1<r<\infty$ and $s>\sigma$. Then
$\n v\n_{L^{r,s}}\leq(\frac{\sigma}{r})^{{1}/{\sigma}-{1}/{s}}\n v\n_{L^{r,\sigma}}$.
\end{lemma}

\begin{Corollary}\label{cld1}
The inequality  \eqref{37} is stronger than the inequality \eqref{str} for $2\leq r<\infty$ and $2\leq\tilde{r}<\infty$. 
\end{Corollary}
\begin{proof}[{\bf Proof}]
By Lemma \ref{cld} above, 
for $1<r,\tilde{r}<\infty$ we have 
$\n v\n_{L^qL^r}=\n v\n_{L^qL^{r,r}}\leq\left({2}/{r}\right)^{{1}/{2}-{1}/{r}}\n v\n_{L^qL^{r,2}}$ and 
$\n F\n_{L^{\tilde{q}'}L^{\tilde{r}',2}}\leq\left({\tilde{r}'}/{\tilde{r}'}\right)^{{1}/{\tilde{r}'}-{1}/{2}}\n F\n_{L^{\tilde{q}'}L^{\tilde{r}',\tilde{r}'}}=\n F\n_{L^{\tilde{q}'}L^{\tilde{r}'}}.$ Since $r\geq2$ and $2\geq\tilde{r}'$, the result follows.
\end{proof}

\subsection{Compactness in Hilbert space and Dislocation}
Here we discuss  relation between the compactness in a Hilbert space $H$ and compactness in its quotient space $H/G$ associated with a 
Dislocation group $G$ in $H$. 
This will be useful in Section \ref{s4}. 
Let $H$ be a Hilbert space and $G$ be a group of operators in $H$. Then we define an equivalence relation $\sim$ on $H$ by $x\sim y$ if $x=gy$ for some $g\in G$. Let $H/G$ denote the quotient space $H/\sim$. Before defining the dislocation group we recall the notions of strong and weak convergence.
\begin{dfn}[Strong Convergence]
We say a sequence $\{g_n\}$ in $G$ converges to $g\in G$ in the strong operator topology if $\n g_nx-gx\n_H\to0$ for all $x\in H$.
\end{dfn}
\begin{dfn}[Weak convergence]
We say a sequence $\{g_n\}$ in $G$ converges to $g\in G$ in the weak operator topology if $\langle g_nx,y\rangle_H\to\langle gx,y\rangle_H$ for all $x,y\in H$.
\end{dfn}
\begin{dfn}[Dislocation group]
Let $H$ be a Hilbert space and $G$ be a group of unitary operators in $H$. Then $G$ is called a dislocation group if, for any $\{g_n\}$ in $G$ not converging to zero in the weak operator topology on $G$, there exists a subsequence $\{g_{n_k}\}$ of $\{g_n\}$ and a non zero $g\in G$ such that $g_{n_k}\to g$ in the strong operator topology on $G$. 
\end{dfn}

Now we state the main result in this subsection without proof relating compactness in $H$ and in $H/G$ which will be used in Subsection \ref{scat} to achieve the `almost' compactness of the flow of a minimal blow-up solution.
\begin{theorem}\label{63}
Let $G$ be a dislocation in a Hilbert space $H$. Then for any compact set $\K$ in the quotient space $H/G$ with the quotient topology, there exists a compact set $K$ in $H$ such that $\K=P(K)$, where $P:H\to H/G$ is the standard canonical projection. 
\end{theorem}

We end this subsection with an example of dislocation to be used in Section \ref{s4}. Let us consider the Hilbert space $\hd$. For $\lambda>0$ let $T_\lambda:\hd\to\hd$ be the unitary operator
\[
T_\lambda f=\frac{1}{\lambda^{(d-2)/2}}f\left(\frac{\cdot}{\lambda}\right),\quad f\in\hd
\] and set
$
G=\{T_\lambda:\lambda>0\}.
$
 Then $G$ is a group of unitary operators in $\hd$ and   we have the following:

\begin{lemma}\label{16}
The group $G$ defined above is a dislocation in the Hilbert space $\hd$.
\end{lemma}

\section{Proof of the Theorems}\label{s3}

This section is divided into three subsections. In Subsection \ref{s31} we prove the inhomogeneous estimates for solutions to \eqref{cp} with $V=0$ whereas in Subsection \ref{s32} we establish the same estimates to a solution to \eqref{cp} with potentials in weak Lebesgue space $L^{d/2,\infty}(\rd)$. 
In Subsection \ref{s33} we proved the estimates for solutions to \eqref{cp} with inverse-square potentials. 

\subsection{Improvement in unperturbed case}\label{s31}
The idea here is to use the $TT^*$ method as in \cite{keel1998endpoint} where the end-point Strichartz estimates are proved. Using the duality it essentially boils down to show that the operator $T:L^{\tilde{q}'}(\R,L^{\tilde{r}',2}(\rd))\times L^{q'}(\R,L^{r',2}(\rd))\to\C$ defined in \eqref{bil} below, is bounded. This, in turn, further reduces to the show the operator \begin{equation*}
\mathcal{T}:=\{T_j\}_{j\in\mathbb{Z}}:L^{\tilde{q}'}(\R,L^{\tilde{r}',2}(\rd))\times L^{q'}(\R,L^{r',2}(\rd))\to l_1^0,
\end{equation*}
 where $T_j$ as in \eqref{lop} below, is bounded. 
 Now we choose $r_0,r_1$ near $r$ and $\tilde{r}_0,\tilde{r}_1$ near $\tilde{r}$ in a judicial way so that applying Lemma \ref{11} below with $(q,\tilde{q},r,\tilde{r})=(q,\tilde{q},r_0,\tilde{r}_0),(q,\tilde{q},r_1,\tilde{r}_0),(q,\tilde{q},r_0,\tilde{r}_1)$ we obtain three different bounds for $\mathcal{T}$ (with appropriate domains and ranges). Then the result follows from interpolation results. Note that in \cite{vilela2007inhomogeneous}, the author varied the exponents $q,\tilde{q}$ and established the weaker estimate \eqref{str} whereas here we vary the exponents $r,\tilde{r}$ and obtain the stronger estimate \eqref{37}.

\begin{proof}[{\bf Proof of Theorem \ref{3'}}]
Note that 
$u(t)=\int_0^te^{i(t-\tau)\Delta}F(\tau,\cdot)d\tau.$
Using $TT^*$ method we need to prove
\begin{equation*}
|T(F,G)|\lesssim\n F\n_{L^{\tilde{q}'}L^{\tilde{r}',2}}\n G\n_{L^{q'}L^{r',2}}
\end{equation*}where $T$ is given by\begin{equation}\label{bil}
T(F,G)=\int_\R\int_{-\infty}^t\langle e^{-i\tau\Delta}F(\tau,\cdot),e^{-it\Delta}G(t,\cdot)\rangle d\tau dt.
\end{equation}
Decomposing $T$ by $T=\sum T_j$ where\begin{equation}\label{lop}
T_j(F,G)=\int_\R\int_{t-2^{j+1}}^{t-2^j}\langle e^{-i\tau\Delta}F(\tau,\cdot),e^{-it\Delta}G(t,\cdot)\rangle d\tau dt,
\end{equation}
it is enough to prove
\begin{equation}\label{10}
\sum_{j\in\mathbb{Z}}|T_j(F,G)|\lesssim\n F\n_{L^{\tilde{q}'}L^{\tilde{r}',2}}\n G\n_{L^{q'}L^{r',2}}.
\end{equation}
Set $\mathcal{T}_{F,G}=\{T_j(F,G)\}$, then \eqref{10} is equivalent with
\begin{equation}\label{20}
\n\mathcal{T}_{F,G}\n_{l_1^0}\lesssim\n F\n_{L^{\tilde{q}'}L^{\tilde{r}',2}}\n G\n_{L^{q'}L^{r',2}}.
\end{equation}

Now we quote a result due to Vilela, see \cite[Lemma 2.2]{vilela2007inhomogeneous}. Using this Lemma for three different choices of $(r,\tilde{r})$ we would get three estimates. These estimates together with Lemmata \ref{18} and \ref{19} would finally imply \eqref{20}.

\begin{lemma}\label{11}
Let $d\geq3$ and $r,\tilde{r}$ be such that $2\leq r,\tilde{r}\leq\infty$ and \begin{equation}
\frac{d-2}{d}\leq\frac{r}{\tilde{r}}\leq\frac{d}{d-2}.
\end{equation}
Then for all $q,\tilde{q}$ satisfying \begin{equation}
\begin{cases}\frac{1}{q}+\frac{1}{\tilde{q}}\leq1\\
\frac{d}{2}\left(\frac{1}{r}-\frac{1}{\tilde{r}}\right)<\frac{1}{\tilde{q}}\quad\text{if }r\leq\tilde{r}\\
\frac{d}{2}\left(\frac{1}{\tilde{r}}-\frac{1}{r}\right)<\frac{1}{q}\quad\text{if }r\geq\tilde{r}
\end{cases}
\end{equation}
the following estimates holds for all $j\in\mathbb{Z}$
\begin{equation}
|T_j(F,G)|\leq c2^{-j\beta(\tilde{q},q,\tilde{r},r)}\n F\n_{L^{\tilde{q}'}L^{\tilde{r}',2}}\n G\n_{L^{q'}L^{r',2}}
\end{equation}
where $\beta(\tilde{q},q,\tilde{r},r)=\left(\frac{1}{\tilde{q}'}-\frac{1}{q}\right)+\frac{d}{2}\left(\frac{1}{\tilde{r}'}-\frac{1}{r}\right)-1.$
\end{lemma}

Let us fix $q,\tilde{q}$ as in \eqref{2} (
this implies ${1}/{q}+{1}/{\tilde{q}}\leq1$ due to \eqref{sc}) and assume we can choose $r_0,\tilde{r}_0,r_1,\tilde{r}_1\geq2$ 
satisfying 
\begin{equation}\label{15}
\beta(\tilde{q},q,\tilde{r}_0,r_1)=\beta(\tilde{q},q,\tilde{r}_1,r_0)\Longleftrightarrow\frac{1}{r_1}-\frac{1}{r_0}=\frac{1}{\tilde{r}_1}-\frac{1}{\tilde{r}_0}
\end{equation} 
\begin{align}\label{23}
\frac{d-2}{d}<\frac{r_j}{\tilde{r_k}}<\frac{d}{d-2},\quad
\left\{
\begin{array}{ll} 
      \frac{d}{2}\left(\frac{1}{r_j}-\frac{1}{\tilde{r}_k}\right)<\frac{1}{\tilde{q}}\quad\text{if }r_j\leq\tilde{r}_k\\
\frac{d}{2}\left(\frac{1}{\tilde{r}_k}-\frac{1}{r_j}\right)<\frac{1}{q}\quad\text{if }r_j\geq\tilde{r}_k
\end{array} ,
\right.
\end{align}for $(j,k)=(0,0),(1,0),(0,1)$,
such that applying Lemma \ref{11},  we achieve
\begin{equation}\label{21}\mathcal{T}:\begin{cases}
L^{\tilde{q}'}L^{\tilde{r}_0'}\times L^{q'}L^{r_0'}\longrightarrow l_\infty^{\beta(\tilde{q},q,\tilde{r}_0,r_0)}=l_\infty^{\beta_0},\\
L^{\tilde{q}'}L^{\tilde{r}_0'}\times L^{q'}L^{r_1'}\longrightarrow l_\infty^{\beta(\tilde{q},q,\tilde{r}_0,r_1)}=l_\infty^{\beta_1},\\
L^{\tilde{q}'}L^{\tilde{r}_1'}\times L^{q'}L^{r_0'}\longrightarrow l_\infty^{\beta(\tilde{q},q,\tilde{r}_1,r_0)}=l_\infty^{\beta_1}.
\end{cases}
\end{equation}
Let us impose the conditions
\begin{align}
(1-\theta)\beta_0+\theta\beta_1=0&\quad\text{for some}\quad0<\theta<1,\label{12}\\
\frac{1}{\tilde{r}}=\frac{1-\theta_0}{\tilde{r}_0}+\frac{\theta_0}{\tilde{r}_1}&\quad\text{for some}\quad0<\theta_0<1,\label{13}\\
\frac{1}{r}=\frac{1-\theta_1}{r_0}+\frac{\theta_1}{r_1}&\quad\text{for some}\quad0<\theta_1<1\label{14},\\
\theta_0+\theta_1=\theta&\label{22}
\end{align}to apply Lemma \ref{18} and Lemma \ref{19}.
Applying Lemma \ref{18} we get
\begin{equation}
\mathcal{T}:(L^{\tilde{q}'}L^{\tilde{r}_0'},L^{\tilde{q}'}L^{\tilde{r}_1'})_{\theta_0,2}\times (L^{q'}L^{r_0'},L^{q'}L^{r_1'})_{\theta_1,2}\longrightarrow (l_\infty^{\beta_0},l_\infty^{\beta_1})_{\theta,1}
\end{equation}which implies $
\mathcal{T}:L^{\tilde{q}'}L^{\tilde{r}',2}\times L^{q'}L^{r',2}\longrightarrow l_1^0
$  (using Lemma \ref{19}, \textcolor{black}{ Minkowski inequality, $q,\tilde{q}\geq2$}) i.e.  \eqref{20}.

Now it is enough to find $r_0,\tilde{r}_0,r_1,\tilde{r}_1>2,\theta_0,\theta_1,\theta$ satisfying  \eqref{15}, \eqref{23}, \eqref{12}, \eqref{13}, \eqref{14} and \eqref{22}.
Since the maps $(x,y)\mapsto\frac{x}{y},(x,y)\mapsto\frac{d}{2}\left(\frac{1}{x}-\frac{1}{y}\right)$ are continuous on $(0,\infty)\times(0,\infty)$, because of \eqref{2}, there exists $\delta>0$ such that
\begin{align*}
\frac{d-2}{d}<\frac{{1}/{r}+a}{{1}/{\tilde{r}}+b}<\frac{d}{d-2},
\end{align*}

\begin{equation*}
\begin{cases}
\frac{d}{2}\left(\frac{1}{r}+a-\frac{1}{\tilde{r}}-b\right)<\frac{1}{\tilde{q}}\quad\text{if }r\leq\tilde{r}\\
\frac{d}{2}\left(\frac{1}{\tilde{r}}+a-\frac{1}{r}-b\right)<\frac{1}{q}\quad\text{if }r\geq\tilde{r}
\end{cases}
\end{equation*}
and $1/r+a,1/\tilde{r}+b>2$ for all $|a|,|b|\leq\delta$.
Set \begin{equation*}
\frac{1}{r_0}=\frac{1}{r}-a,\quad\frac{1}{r_1}=\frac{1}{r}+b,\quad\frac{1}{\tilde{r}_0}=\frac{1}{\tilde{r}}-a,\quad\frac{1}{\tilde{r}_1}=\frac{1}{\tilde{r}}+b
\end{equation*}with \begin{equation*}0<a,b<
\begin{cases}
\min\{\delta,\frac{1}{r},\frac{1}{\tilde{r}},\frac{1}{2}-\frac{1}{r},\frac{1}{2}-\frac{1}{\tilde{r}},\frac{1}{2}\left|\frac{1}{r}-\frac{1}{\tilde{r}}\right|\}\quad\text{if }r\neq\tilde{r}\\
\min\{\delta,\frac{1}{r},\frac{1}{\tilde{r}},\frac{1}{2}-\frac{1}{r},\frac{1}{2}-\frac{1}{\tilde{r}},\frac{1}{2}\}\quad\quad\quad\quad\ \text{if }r=\tilde{r}.
\end{cases}
\end{equation*}
Then \eqref{15} and \eqref{23} are satisfied. Because of \eqref{13} and \eqref{14} we have $a(1-\theta_0)=b\theta_0$ and $a(1-\theta_1)=b\theta_1$. Adding them we have $a(2-\theta)=b\theta$ using \eqref{22}. Therefore we have \begin{equation}\label{27}
\theta=\frac{2a}{a+b}.
\end{equation}
Subtracting $a(1-\theta_0)=b\theta_0$ from $a(1-\theta_1)=b\theta_1$ we get $\theta_0=\theta_1$ and therefore $\theta_0=\theta_1=\frac{a}{a+b}$. Since we require $0<\theta<1$ we choose $a<b$.

Note that \eqref{12} is equivalent to
\begin{align}\label{17}
\left(\frac{1}{\tilde{q}'}-\frac{1}{q}\right)+\frac{d}{2}\left(\frac{1}{\tilde{r}_0'}-\frac{1}{r_0}\right)+\theta\frac{d}{2}\left[\left(\frac{1}{\tilde{r}_0'}-\frac{1}{r_1}\right)-\left(\frac{1}{\tilde{r}_0'}-\frac{1}{r_0}\right)\right]=1
\end{align}
Now using \eqref{2} we have 
$
\frac{1}{\tilde{q}'}-\frac{1}{q}=1-\frac{1}{\tilde{q}}-\frac{1}{q}=1-\frac{d}{2}\left(1-\frac{1}{r}-\frac{1}{\tilde{r}}\right)
$
and therefore \eqref{17} is equivalent with
\begin{align*}
&\frac{d}{2}\left(\frac{1}{\tilde{r}_0'}-\frac{1}{r_0}\right)+\theta\frac{d}{2}\left(\frac{1}{r_0}-\frac{1}{r_1}\right)=\frac{d}{2}\left(1-\frac{1}{r}-\frac{1}{\tilde{r}}\right)\\
\Longleftrightarrow&\frac{1}{r_0}+\frac{1}{\tilde{r}_0}=\frac{1}{r}+\frac{1}{\tilde{r}}+\theta\left(\frac{1}{r_0}-\frac{1}{r_1}\right)
\end{align*}
and by our choice of $r_0,\tilde{r}_0,r_1,\tilde{r}_1,\theta_0,\theta_1,\theta$ this is equivalent to $2a=\theta(a+b)$ which is equivalent to \eqref{27}.  
\end{proof}

\subsection{Potential in $L^{d/2,\infty}(\rd)$}\label{s32}
\begin{proof}[{\bf Proof of Theorem \ref{1}}]
 Let us split $u$ as $u=u_1+u_2$ where $u_1,u_2$ satisfy\\ \begin{minipage}{.5\linewidth}
\[
\begin{cases} i\partial_t u_1+\Delta u_1=F\\
u_1(0,\cdot)=0
\end{cases},
\]
\end{minipage}
\begin{minipage}{.3\linewidth}
\[
\begin{cases} i\partial_t u_2+\Delta u_2=-Vu\\
u_2(0,\cdot)=0.
\end{cases}
\]
\end{minipage}\\
Let $r,\tilde{r},q,\tilde{q}$ satisfy \eqref{2}.
Using Theorem \ref{3'} for exponents $(q,r),(\tilde{q},\tilde{r})$ we have that
\begin{equation*}\label{6}
\n u_1\n_{L^qL^{r,2}}\leq c_s\n F\n_{L^{\tilde{q}'}L^{\tilde{r}',2}}
\end{equation*}
and for exponent $(q,r),(q',(\frac{dr}{2r+d})')$ we have 
\begin{equation*}\label{7}
\n u_2\n_{L^qL^{r,2}}\leq c_s\n Vu\n_{L^qL^{\frac{dr}{2r+d},2}}
\end{equation*}provided we farther assume
\begin{equation*}\begin{cases}
d\left(\frac{1}{r}-\frac{1}{2^*}\right)<\frac{1}{q'}\quad \text{for }\frac{2(d-1)}{d-2}<r\leq2^*,\\
d\left(\frac{1}{2^*}-\frac{1}{r}\right)<\frac{1}{q}\quad \text{for }2^*\leq r<\frac{2^*(d-1)}{d-2}.
\end{cases}
\end{equation*}
Now using H\"{o}lder inequality for Lorentz spaces (see Lemma \ref{lh}) we have
\begin{equation*}
\n Vu\n_{L^qL^{\frac{dr}{2r+d},2}}\leq (\frac{dr}{2r+d})'\n V\n_{L^{\infty}L^{d/2,\infty}}\n u\n_{L^qL^{r,2}}
\end{equation*}
and therefore 
\begin{align*}
\n u\n_{L^qL^{r,2}}&\leq\n u_1\n_{L^qL^{r,2}}+\n u_2\n_{L^qL^{r,2}}\leq c_s\left(\n F\n_{L^{\tilde{q}'}L^{\tilde{r}',2}}+\n Vu\n_{L^{\tilde{q}'}L^{\tilde{r}',2}}\right)\\
&\leq c_s\left(\n F\n_{L^{\tilde{q}'}L^{\tilde{r}',2}}+(\frac{dr}{2r+d})'\n V\n_{L^{\infty}L^{d/2,\infty}}\n u\n_{L^qL^{r,2}}\right).
\end{align*}
Then we have that 
\begin{equation*}\label{9}
\n u\n_{L^qL^{r,2}}\leq\frac{c_s}{1-c_s(\frac{dr}{2r+d})'\n V\n_{L^\infty L^{d/2,\infty}}}\n F\n_{L^{\tilde{q}'}L^{\tilde{r}',2}}
\end{equation*}provided  $c_s(\frac{dr}{2r+d})'\n V\n_{L^\infty L^{d/2,\infty}}<1$.
\end{proof}

\begin{proof}[{\bf Proof of Theorem \ref{1'}}]

{\it Case I:} Here we prove the result assuming the \textcolor{black}{first set of conditions}. Multiplying the equation \eqref{cp} by $\bar{u}$ and integrating by parts we get 
\[
i\int_{\rd}\partial_tu\bar{u}-\int_{\rd}|\nabla u|^2+\int_{\rd}V|u|^2=\int_{\rd}F\bar{u}.
\] 
Taking imaginary part of both side we get $
{\rm Re} \left(\int_{\rd}\partial_tu\bar{u}\right)={\rm Im}\left(\int_{\rd}F\bar{u}\right)
$. Cauchy-Schwartz inequality now implies $\partial_t\n u(t)\n^2\leq2\n u(t)\n\n F(t)\n$ which in turn gives (after cancelling one $\n u(t)\n$ from both side and then integrating in time on $[0,t]$)
\begin{equation}\label{29'}
\n u\n_{L^\infty L^{2,2}}\lesssim\n F\n_{L^1L^{2,2}}
\end{equation}(see proof of Proposition 3 in \cite{pierfelice2006strichartz} for details).
Now we would like to have the estimate
\textcolor{black}{
\begin{equation}\label{46}
\n u\n_{L^{2}L^{r_1,2}}\lesssim\n F\n_{L^{\tilde{q}_1'}L^{\tilde{r}_1',2}}
\end{equation}}
using Theorem \ref{1} for appropriate $q_1,\tilde{q}_1,\tilde{r}_1$. \\
\textcolor{black}{
Choose $0\leq\theta\leq1$ so that $\frac{1}{q}=({1-\theta})0+\frac{\theta}{2}\Longleftrightarrow\theta=\frac{2}{q}$, then take $\tilde{q}_1\geq2$, $2_*<r_1<2_*^*$, $\tilde{r}_1>2$,  so that $\tilde{q}_1=\theta\tilde{q}_1$, $\frac{1}{{r}}=\frac{1-\theta}{2}+\frac{\theta}{{r}_1}$ and $\frac{1}{\tilde{r}}=\frac{1-\theta}{2}+\frac{\theta}{\tilde{r}_1}$. 
Let us now verify the conditions in Theorem \ref{1} so that \eqref{46} holds. Note that by direct computation we have 
\begin{enumerate}
\item[•] $\frac{1}{2}+\frac{1}{\tilde{q}_1}+\frac{d}{2}\left(\frac{1}{r_1}+\frac{1}{\tilde{r}_1}\right)=\frac{d}{2}\Longleftrightarrow\frac{1}{q}+\frac{1}{\tilde{q}}+\frac{d}{2}\left(\frac{1}{r}+\frac{1}{\tilde{r}}\right)=\frac{d}{2}\Longleftrightarrow\eqref{sc}$
\item[•] $\frac{d}{d-2}<\frac{\tilde{r}_1}{r_1}<\frac{d}{d-2}\Longleftrightarrow \frac{1}{2}-\frac{1}{\tilde{r}}-\frac{1}{q}<\frac{d}{2}\left(\frac{1}{r}-\frac{1}{\tilde{r}}\right)<\frac{1}{r}-\frac{1}{2}+\frac{1}{q}$
\item[•] $\tilde{q}_1\geq2\Longleftrightarrow \tilde{q}\geq q$
\item[•] $\frac{d}{2}\left(\frac{1}{{r}_1}-\frac{1}{\tilde{r}_1}\right)<\frac{1}{\tilde{q}_1}\text{ if }r_1\leq\tilde{r}_1\Longleftrightarrow\frac{d}{2}\left(\frac{1}{{r}}-\frac{1}{\tilde{r}}\right)<\frac{1}{\tilde{q}} \text{ if }r\leq\tilde{r}$, $\frac{d}{2}\left(\frac{1}{\tilde{r}_1}-\frac{1}{r_1}\right)<\frac{1}{q_1}\text{ if }r_1\geq\tilde{r}_1\Longleftrightarrow\frac{d}{2}\left(\frac{1}{\tilde{r}}-\frac{1}{r}\right)<\frac{1}{q}\text{ if }r\geq\tilde{r}\Longleftrightarrow\eqref{2}$
\item[•] $2_*<r_1<2_*^*\Longleftrightarrow\frac{2q(d-1)}{q(d-1)-2}<r<\frac{2qd(d-1)}{qd(d-1)-6d+8}$ and $\tilde{r}_1>2\Longleftrightarrow\tilde{r}>2$
\item[•] $d\left|\frac{1}{r_1}-\frac{1}{2^*}\right|<\frac{1}{2}\Longleftrightarrow\left|1-\frac{q}{2d}\big(\frac{1}{2}-\frac{1}{r}\big)\right|<\frac{1}{2}.$
\end{enumerate}}
Now for $c_s2^*\n V\n_{L^\infty L^{d/2,\infty}}<1$, the above conditions ensures \eqref{46}. 
 Interpolating (see Lemma \ref{int}) \eqref{29'} and \eqref{46}, we get the result.\\

%
%
{\it Case II:} Let us assume the \textcolor{black}{second set of conditions}. As $(\infty,2),(2,2^*)$ are admissible pairs, by \cite[Theorems 1, 3]{pierfelice2006strichartz}
, we have
\begin{align}\label{44'}
\n u\n_{L^{\infty}L^{2,2}}\lesssim\n F\n_{L^2L^{{2^*}',2}}
\end{align} for $c_s2^*\n V\n_{L^{d/2,\infty}}<1$.
Here again we would like to have the estimate of the form 
\textcolor{black}{
\begin{equation}\label{46'}
\n u\n_{L^2L^{r_1,2}}\lesssim\n F\n_{L^{\tilde{q}_1'}L^{\tilde{r}_1',2}}
\end{equation}}
using Theorem \ref{1} for appropriate $q_1,\tilde{q}_1,\tilde{r}_1$. \\
\textcolor{black}{
Choose $0\leq\theta\leq1$ so that $\frac{1}{q}=({1-\theta})0+\frac{\theta}{2}\Longleftrightarrow\theta=\frac{2}{q}$, then take $\tilde{q}_1\geq2$, $2_*<r_1<2_*^*$, $\tilde{r}_1>2$,  so that $\frac{1}{\tilde{q}}=\frac{1-\theta}{2}+\frac{\theta}{\tilde{q}_1}$, $\frac{1}{{r}}=\frac{1-\theta}{2}+\frac{\theta}{{r}_1}$ and $\frac{1}{\tilde{r}}=\frac{1-\theta}{2^*}+\frac{\theta}{\tilde{r}_1}$. 
Then again by direct computation we have
\begin{enumerate}
\item[•] $\frac{1}{2}+\frac{1}{\tilde{q}_1}+\frac{d}{2}\left(\frac{1}{r_1}+\frac{1}{\tilde{r}_1}\right)=\frac{d}{2}\Longleftrightarrow\frac{1}{q}+\frac{1}{\tilde{q}}+\frac{d}{2}\left(\frac{1}{r}+\frac{1}{\tilde{r}}\right)=\frac{d}{2}\Longleftrightarrow\eqref{sc}$
\item[•] $\frac{d}{d-2}<\frac{\tilde{r}_1}{r_1}<\frac{d}{d-2}\Longleftrightarrow \frac{d-1}{d}\left(1-\frac{2}{q}\right)-\frac{1}{\tilde{r}}<\frac{d}{2}\left(\frac{1}{{r}}-\frac{1}{\tilde{r}}\right)<\frac{1}{r}$
\item[•] $\tilde{q}_1\geq2\Longleftrightarrow\tilde{q}\geq2$
\item[•]$\frac{d}{2}\left(\frac{1}{{r}_1}-\frac{1}{\tilde{r}_1}\right)<\frac{1}{\tilde{q}_1}\text{ if }r_1\leq\tilde{r}_1\Longleftrightarrow\frac{d}{2}\left(\frac{1}{{r}}-\frac{1}{\tilde{r}}\right)<\frac{1}{\tilde{q}} \text{ if }\frac{1}{r}-\frac{1}{\tilde{r}}\geq\frac{1}{d}\left(1-\frac{2}{q}\right)$, $\frac{d}{2}\left(\frac{1}{\tilde{r}_1}-\frac{1}{r_1}\right)<\frac{1}{q_1}\text{ if }r_1\geq\tilde{r}_1\Longleftrightarrow\frac{d}{2}\left(\frac{1}{\tilde{r}}-\frac{1}{r}\right)<\frac{1}{q}\text{ if }\frac{1}{r}-\frac{1}{\tilde{r}}\leq\frac{1}{d}\left(1-\frac{2}{q}\right)$
\item[•] $2_*<r_1<2_*^*\Longleftrightarrow\frac{2q(d-1)}{q(d-1)-2}<r<\frac{2qd(d-1)}{qd(d-1)-6d+8}$ and $\tilde{r}_1>2\Longleftrightarrow\tilde{r}>2$
\item[•] $d\left|\frac{1}{{r}_1}-\frac{1}{2^*}\right|<\frac{1}{2}\Longleftrightarrow\left|1-\frac{q}{2d}\big(\frac{1}{2}-\frac{1}{r}\big)\right|<\frac{1}{2}.$
\end{enumerate}}

The above set of assumption together with $c_s2^*\n V\n_{L^\infty L^{d/2,\infty}}<1$ imply \eqref{46'}.
Now we interpolate \eqref{44'} and \eqref{46'} to get the result.
\end{proof}

\subsection{Inverse square potential}\label{s33}

\begin{proof}[{\bf Proof of Theorem \ref{56}}]
Note we have to prove the case (iii) only. We prove the result using interpolation twice in two steps.\\

{\it Step I:} Set $\frac{1}{q_0}=\frac{1}{\tilde{q}_0}=\frac{1}{2},\frac{1}{r_0}=\frac{d-2\sigma}{2d}, \frac{1}{\tilde{r}_0}=\frac{d-2(2-\sigma)}{2d}$, $\frac{1}{\tilde{q}_1}=\frac{1}{2}$ and $\frac{1}{\tilde{r}_1}=\frac{d-2}{2d}$. 
From Theorem \ref{56} (ii) 
we have \begin{equation}\label{58}
\left\n\int_0^te^{i(t-\tau)\la}F(\tau)d\tau\right\n_{L^2L^{\frac{2d}{d-2\sigma}}}\lesssim\n F\n_{L^2L^{(\frac{2d}{d-2(2-\sigma)})'}}
\end{equation} for appropriate $\sigma$.
Let us choose $\theta$ so that 
\begin{align}\label{57}
\frac{1}{\tilde{r}}=\frac{1-\theta}{\tilde{r}_0}+\frac{\theta}{\tilde{r}_1}.
\end{align} 
Note that
$
\frac{1}{\tilde{r}_0}-\frac{1}{\tilde{r}}=\frac{d-2(2-\sigma)}{2d}-\frac{d-2(2-s)}{2d}=\frac{\sigma-s}{d}
$ and
$
\frac{1}{\tilde{r}_0}-\frac{1}{\tilde{r}_1}=\frac{d-2(2-\sigma)}{2d}-\frac{d-2}{2d}=\frac{2-2(2-\sigma)}{2d}=\frac{1-(2-\sigma)}{d}=\frac{\sigma-1}{d}.
$
These together with \eqref{57} imply $\theta=\big(\frac{1}{\tilde{r}}-\frac{1}{\tilde{r}_0}\big)/\big(\frac{1}{\tilde{r}_1}-\frac{1}{\tilde{r}_0}\big)=\frac{\sigma-s}{\sigma-1}$.
In order to make $\theta\in(0,1)$ we must have \begin{equation}\label{64}
1<s<\sigma \quad {\rm or } \quad\sigma<s<1.
\end{equation}  Set $q_1$ so that $
 \frac{1}{q}=\frac{1-\theta}{q_0}+\frac{\theta}{q_1}.
 $
 To make $q_1>0$ we need \begin{align}\label{60}
 \frac{1}{q_1}=\frac{1}{\theta}\left[\frac{1}{q}-(1-\theta)\frac{1}{2}\right]>0\Longleftrightarrow\frac{\gamma}{2}>\frac{s-1}{\sigma-1}\frac{1}{2}\Longleftrightarrow\gamma>\frac{s-1}{\sigma-1}.
 \end{align}
 Then  
 $
 \frac{1}{2}-\frac{1}{q_1}=\frac{1}{\theta}\left(\frac{1}{2}-\frac{1}{q}\right)\geq0\Longleftrightarrow q_1\geq2
 $ as $q\geq2$. Now choose $r_1$ so that $(q_1,r_1)$ is an admissible pair. 
 Then by  Theorem \ref{56} (i) we have \begin{equation}\label{59}
\left\n\int_0^te^{i(t-s)\la}F(s)ds\right\n_{L^{q_1}L^{r_1}}\lesssim\n F\n_{L^2L^{{2^*}'}}.
\end{equation}
Interpolating (see Lemma \ref{int}) of \eqref{58} and \eqref{59} we have 
\begin{equation}
\left\n\int_0^te^{i(t-\tau)\la}F(\tau)d\tau\right\n_{L^{2/\gamma}L^{\frac{2d}{d-2(s+\gamma-1)}}}\lesssim\n F\n_{L^2L^{(\frac{2d}{d-2(2-s)})'}}
\end{equation}
where $s\in A_{a,\gamma}=\left(1-\frac{d-2}{2(d-1)}\gamma,1+\frac{d-2}{2(d-1)}\gamma\right)\cap R_{a,\gamma}$.
Note that \eqref{60} is equivalent to
\begin{align*}\begin{cases}
s<1+(\sigma-1)\gamma\text{ if }\sigma>1\\
s>1-(1-\sigma)\gamma\text{ if }\sigma<1.
\end{cases}
\end{align*}This ensures  \eqref{64} and sets the conditions $s\in A_{a,\gamma}$.\\

{\it Step II:} Set $\frac{1}{q_0}=\frac{2}{\gamma},\frac{1}{r_0}=\frac{d-2(\sigma+\gamma-1)}{2d},\frac{1}{\tilde{q}_0}=\frac{1}{2},\frac{1}{\tilde{r}_0}=\frac{d-2(2-\sigma)}{2d},\frac{1}{q_1}=\frac{\gamma}{2}$ and $\frac{1}{r_1}=\frac{d-2\gamma}{2d}$. From Step I we have 
\begin{equation}\label{58'}
\left\n\int_0^te^{i(t-\tau)\la}F(\tau)d\tau\right\n_{L^{2/\gamma}L^{\frac{2d}{d-2(\sigma+\gamma-1)}}}\lesssim\n F\n_{L^2L^{(\frac{2d}{d-2(2-\sigma)})'}}
\end{equation} for appropriate $\sigma$. Set $\theta$ so that 
$
\frac{1}{r}=\frac{1-\theta}{r_0}+\frac{\theta}{r_1}.
$
Then $\frac{1}{r}-\frac{1}{r_0}
=\frac{\sigma-s}{d}$, $\frac{1}{r_1}-\frac{1}{r_0}
=\frac{\sigma-1}{d}$ and hence $\theta=\frac{\sigma-s}{\sigma-1}$. 
 Take $\tilde{q}_1$ so that $\frac{1}{\tilde{q}}=\frac{1-\theta}{\tilde{q}_0}+\frac{\theta}{\tilde{q}_1}$. To make $\tilde{q}_1>0$ as before we need  $\tilde{\gamma}>\frac{s-1}{\sigma-1}$. This ensures $\theta
\in(0,1)$ and sets the condition $s\in A_{a,\gamma\tilde{\gamma}}$. At last choose $\tilde{r}_1$ so that $(\tilde{q}_1,\tilde{r}_1)$ is an admissible pair. Then by Theorem \ref{56} (i) we have \begin{equation}\label{59'}
\left\n\int_0^te^{i(t-s)\la}F(s)ds\right\n_{L^{q_1}L^{r_1}}\lesssim\n F\n_{L^{\tilde{q}_1'}L^{{\tilde{r}_1}'}}.
\end{equation}
 Now the theorem follows from interpolation of \eqref{58'} and \eqref{59'}.
\end{proof}

\section{Application}\label{s4}
In this section, we study the scattering solutions of the Cauchy problem
\[\tag{NLS$_a$}\label{61}
i\frac{\partial}{\partial t}u(t,x)+ \mathcal{L}_a u(t,x)+ |u(t,x)|^{\frac{4}{d-2}}u(t,x) =0, \quad u(t_0,x)= u_0(x).
\]
We show that as an application of Theorem \ref{56} (iii), we can establish a stability result for this problem with $a\neq0$, similar to that of \cite[Theorem 2.14]{kenig2006global} for the case $a=0$. This stability result in turn will establish the existence of scattering solution for radial data in dimensions $3,4$ and $5$ by proceeding exactly as in Kenig and Merle \cite{kenig2006global}. In fact, when this project was in its final stage, we came across the very recent work of Yang \cite{yang2020scattering0} where the same result has been established using a different argument.  Therefore our work serves as an alternative proof of Theorem \ref{mt}.


In Subsection \ref{ss}, we present the stability of solutions to \eqref{61} in detail, and in Subsection \ref{scat} we outline the proof of the scattering result without details as the proof deviates very little from that of \cite{kenig2006global}.
\subsection{Stability of Solution.}\label{ss}
 Let  $I$ be an open interval in $\R, t_0 \in I$ and $u_0\in \hd$ . We say that  $u\in C(I, \dot{H}^1(\mathbb R^d))$ is a solution of \eqref{61} if  $\n\nabla u \n_{W(\tilde{I})} <\infty$ for all $\tilde{I}\subset \subset I$ and satisfy the integral equation
$$  u(t)= e^{i(t-t_0)\la} u_0+ i\int_{t_0}^{t}S_a(t-\tau)f(u(\tau))d\tau,$$
with  $f(u)= |u|^{\frac{4}{d-2}}u$. 
Then proceeding exactly as in the proof of Theorem 2.5 in \cite{kenig2006global}  by using Strichartz estimates with inverse square potential i.e. Theorem \ref{56} (and Corollary \ref{st_a}) we can establish the following local existence theorem.

\begin{prp}[Local existence]\label{gsd} Let $d\in\{3,4,5\}$ and  $a < \left(\frac{d-2}{2}\right)^2-\left(\frac{d-2}{d+2}\right)^2$. 
Then for every $A>0$ there exists  $\delta = \delta (A)>0$ such that for any interval $I\subset \R$ containing $t_0$ and $u_0\in \hd$ satisfying  $\|u_0\|_{\dot{H}^1}<A$ and $\|S_a(t-t_0)u_0\|_{S(I)} < \delta$,
 the Cauchy problem \eqref{61} has a unique solution in  $I$  with $\| \nabla u \|_{W(I)} < \infty,\quad  \|u\|_{S(I)} \leq 2 \delta$ .
Moreover, if $u_{0,k} \to u_0$ in  $\dot{H}^1(\mathbb R^d)$ , the corresponding solutions $u_k\to u$ in
 $C(I, \dot{H}^{1}(\mathbb R^d))$.
\end{prp}

\begin{rmk}
{\rm Note that in the above result we have further restriction on $a$. This restriction comes to achieve equivalence of the norms $\n\cdot\n_{\dot{W}_a^{1,r}}$  and $\n\cdot\n_{\dot{W}^{1,r}}$ with $r=\frac{2d(d+2)}{d^2+4}$ (which would be used to prove Proposition \ref{gsd}, in a way similar to the proof of \cite[Theorem 2.5]{kenig2006global}),  see Corollary \ref{st_a} and Lemma \ref{NormEq}.
}\end{rmk}
Using the above Proposition, we can define the maximal interval of existence. It is easy to see from  Proposition \ref{gsd} by using the Sobolev inequality that \eqref{61} has a global Solution when the initial data is small enough. Also following the very same arguments as \cite[Lemma 2.11]{kenig2006global} we have,
\begin{lemma}[Finite time blow-up criterion]\label{sbu} Let $I=(-T_-(u_0),T_+(u_0))$ be the maximal interval of existence of solution to \eqref{61}.   If $T_{+}(u_0)< \infty,$ then
$\|u\|_{S([t_0, T_{+}(u_0)])}= \infty.$
A similar result holds for $T_{-}(u_ 0 ).$
\end{lemma}
Now we can state the main theorem of this subsection:
\begin{theorem}[{\bf Long time perturbation}]\label{ltp}  Let $d\in\{3,4,5\}$, $a<\left(\frac{d-2}{2}\right)^2-\left(\frac{d-2}{d+2}\right)^2$ and $I$ be an open interval in $\R$ containing $t_0$. Let $\widetilde{u}$ be defined on $I \times \mathbb R^d$  and satisfy  $\sup_{t\in I} \|\widetilde{u}(t)\|_{\dot{H}^1} \leq A,  \|\widetilde{u}\|_{S(I)} \leq M$ for some constants $M, A>0$. Assume that $\widetilde{u}$ satisfies
$i\partial_t\widetilde{u} +\mathcal{L}_a\widetilde{u} + f (\widetilde{u})=g$, i.e.,
$$ \widetilde{u}(t)=S_a(t-t_0)\widetilde{u}(t_0)+ i \int_{t_0}^tS_a(t-\tau) (f(\widetilde{u}(\tau))-g(\tau))d\tau =0.$$
Then for every $A'>0$, there exists  $\epsilon_0 = \epsilon_0 (M,A, A', d)>0$ such that whenever
$$
\|u_0-\widetilde{u}(t_0)\|_{\dot{H}^1} \leq A',\quad \|\nabla g \|_{L^{2} (I,L^{\frac{2d}{d+2}})} \leq \epsilon,\quad  \|S_a(t-t_0)[u_0-\widetilde{u}(t_0)]\|_{S(I)} \leq \epsilon
$$
for some $0< \epsilon< \epsilon_0,$ then the Cauchy Problem \eqref{61}
has a solution $u$ defined on $I$ satisfying the estimate
 \[\|u\|_{S(I)} \leq  C(M, A, A', d)\quad\text{and}\quad\|u(t)-\tilde{u}(t)\|_{\dot{H}^1} \leq C(A, M, d) (A' + \epsilon) \ \ \text{for all} \ t\in I.\] 
\end{theorem} 
%
%
%

\begin{proof}[{\bf Proof}]
First note that for any $u_0$ as in the statement of the theorem, the Cauchy problem \eqref{61} has a solution in a maximal interval of existence by Proposition \ref{gsd}. We prove that this solution satisfies the required a priori estimates. By blow-up alternative  i.e. Lemma \ref{sbu}, this will immediately imply that solution has to exist in all of $I$ as  $\|u\|_{S(I)} \leq  C(M, A, A', d)< \infty$.\\ 

\noindent
{\it STEP I:}
Let us show that $\n\nabla\widetilde{u}(t)\n_{W(I)}\leq M'=M'(A,M,d)<\infty$.  

For $\eta>0$ split $I$ into $\gamma=\gamma(M,\eta)$ intervals $I_1,I_2,\cdots,I_\gamma$ so that $\n \widetilde{u}(t)\n_{S(I_j)}\leq\eta$ for $j=1,2,\cdots,\gamma$. Then 
$$\widetilde{u}(t)=S_a(t-t_j)\widetilde{u}(t_j)+i\int_{t_j}^tS_a(t-\tau)f\circ\widetilde{u}(\tau)d\tau+i\int_{t_j}^tS_a(t-\tau) g(\tau)d\tau$$ for some $t_j\in I_j$ fixed. Then
\begin{align*}
\n \nabla\widetilde{u}\n_{W(I_j)}&\leq  cA+c\n \widetilde{u}\n_{S(I_j)}^{\frac{4}{d-2}}\n \nabla \widetilde{u}\n_{W(I_j)}+c\n \nabla g\n_{L^2(I_j,L^{\frac{2d}{d+2}}(\rd))}\\
&\leq cA+c\eta^{\frac{4}{d-2}}\n \nabla \widetilde{u}\n_{W(I_j)}+c\n \nabla g\n_{L^2(I_j,L^{\frac{2d}{d+2}}(\rd))}\leq c(A+\varepsilon)+\frac{1}{2}\n \nabla \widetilde{u}\n_{W(I_j)}
\end{align*}
choosing $\eta=\eta(d)>0$ small enough. Hence we have $\lVert \nabla\widetilde{u}\rVert_{W(I_j)}\leq 2c(A+\varepsilon)$ consequently by taking $\varepsilon_0\leq1$, we have $\lVert \nabla\widetilde{u}\rVert_{W(I)}\leq 2\gamma(\eta(d),M)c(A+1)=:M'(A,M,d).$ \\

\noindent
{\it STEP II:}
A priori estimate. 

 Here we follow  \cite{kenig}, where the case $a=0$ is dealt. Let us set  $q,r,\tilde{q},\tilde{r}$ by $q=\frac{2(d+2)}{d-2}$, $\frac{1}{r}=\frac{d-2}{2(d+2)}+\frac{\alpha}{d}$, $\tilde{q}=2$ and $\frac{1}{\tilde{r}}=\frac{d^2+2(1-\alpha)d-4\alpha}{2d(d+2)}$. If we write $\frac{1}{q}=\frac{\gamma}{2},\frac{1}{\tilde{q}}=\frac{\tilde{\gamma}}{2},\frac{1}{\tilde{r}}=\frac{d-2(2-s)}{2d}$ then we have $\gamma=\frac{d-2}{d+2}<1,\tilde{\gamma}=1$ and $s-1=1-\alpha$. Since Theorem \ref{56} (iii) is valid for $s$ in a neighbourhood of $1$, we conclude  
\begin{equation}\label{istri}
\left\n\int_0^tS_a(t-\tau)h(\tau)d\tau\right\n_{L^qL^r}\lesssim\n h\n_{L^{\tilde{q}'}L^{\tilde{r}'}}
\end{equation} is valid for $0<\alpha<1$ close enough to 1.
Then we have by fractional Hardy inequalities  
\begin{equation}\label{fhi}
\n f\n_{S(I)}\lesssim\n D^\alpha f\n_{L^qL^r}\lesssim\n\nabla f\n_{W(I)}
\end{equation}by interpolation
\begin{equation}\label{fri}
\n D^\alpha f\n_{L^q(I,L^r)}\lesssim\n f\n_{S(I)}^{1-\alpha}\n\nabla f\n_{W(I)}^\alpha
\end{equation} by Holder 
\begin{equation}
\n |u|^{4/(d-2)}D^\alpha u\n_{L^{\tilde{q}'}L^{\tilde{r}'}}\leq\n u\n_{S(I)}^{4/(d-2)}\n D^\alpha u\n_{L^qL^r}.
\end{equation}

Let $\eta>0$. Again split $I$ into $l=l(M,M',\eta)$ intervals $I_0,I_1,\cdots,I_{l-1}$ with $I_j=[t_j, t_{j+1}]$ so that $\lVert \widetilde{u}\rVert_{S(I_j)}\leq\eta$ and $\lVert D^\alpha\widetilde{u}\rVert_{L^q(I_j,L^r)}\leq\eta$ for $j=0,1,\cdots,l-1$.
Let us write $u=\widetilde{u}+w$. Then $w$ solves 
$$i\partial_tw+\la w+f(\widetilde{u}+w)-f(\widetilde{u})=-g$$with $w(t_0)=u_0-\widetilde{u}(t_0)$ if $u$ solves \eqref{61}.
Now in order to solve for $w,$  we need to solve, in $I_j,$ the integral equation
 \begin{equation}\label{p1}
 w(t)=S_a(t-t_j)e^{i(t-t_j) \mathcal{L}_a}w(t_j)+i\int_{t_j}^tS_a(t-\tau)[f(\widetilde{u}+w)-f(\widetilde{u})](\tau)d\tau+i\int_{t_j}^tS_a(t-\tau) g(\tau)d\tau.
 \end{equation}

Put $B_j=\n D^\alpha w\n_{L^q(I_j,L^r)},\gamma_j=\n D^\alpha e^{i (t-t_j) \mathcal{L}_a}w(t_j)\n_{L^q(I_j,L^r)}+c\varepsilon$ and $N_j(w,\widetilde{u})=\n D^\alpha[(f\circ(\widetilde{u}+w))-(f\circ\widetilde{u} )]\n_{L^{\tilde{q}'}(I_j,L^{\tilde{r}'})}$.
Then by \eqref{istri} (see also Remark \ref{rm3} below also)
\begin{align*}
B_j\leq\gamma_j+cN_j(w,\widetilde{u}).
\end{align*}
Now by fractional Leibnitz and chain rule 
\begin{align*}
N_j(w,\widetilde{u})&\lesssim\left(\lVert \widetilde{u}\rVert_{S(I_j)}^{\frac{4}{d-2}}+\lVert w\rVert_{S(I_j)}^{\frac{4}{d-2}}\right)\lVert D^\alpha w\rVert_{L^q(I_j,L^r)}\\
&\quad+\lVert w\rVert_{S(I_j)}\left(\lVert \widetilde{u}\rVert_{S(I_j)}^{\frac{6-d}{d-2}}+\lVert w\rVert_{S(I_j)}^{\frac{6-d}{d-2}}\right)\left(\lVert D^\alpha \widetilde{u}\rVert_{L^q(I_j,L^r)}+\lVert D^\alpha w\rVert_{L^q(I_j,L^r)}\right).
\end{align*}
Therefore $
B_j\leq\gamma_j+c\eta^{\frac{4}{d-2}}B_j+cB_j^{\frac{d+2}{d-2}}
$ and choosing $\eta>0$ small 
\begin{align*}
B_j\leq 2\gamma_j+cB_j^{\frac{d+2}{d-2}}=2\gamma_j+cB_j^{\frac{4}{d-2}}B_j.
\end{align*}
This implies if $B_j\leq\left(\frac{1}{2c}\right)^{\frac{d-2}{4}}=:c_0$ (so that $cB_j^{\frac{4}{d-2}}\leq\frac{1}{2}$) then $B_j\leq 4\gamma_j$.
Hence we have \begin{align*}
\n\nabla w\n_{W(I_j)}\leq 4\left(\n e^{-i (t-t_j) \mathcal{L}_a}w(t_j)\n_{W(I)}+c\varepsilon\right)\quad{\rm provided }\ B_j\leq c_0.
\end{align*}
Now 
put $t=t_{j+1}$ in the integral formula \eqref{p1},  and apply  $S_a(t-t_{j+1})$  to we obtain
\begin{align*}
S_a(t-t_{j+1})w(t_{j+1})  &=S_a(t-t_j)  w(t_j) +i  \int_{t_j}^{t_{j+1}}S_a(t-\tau) [f(\tilde{u}+w)- f(\tilde{u})](\tau) d\tau \\
&\quad + i \int_{t_j}^{t_{j+1}}S_a(t-\tau) g(\tau) d\tau.
\end{align*}
Therefore as before provided $B_j \leq c_0$ we have
\begin{align*}
\n D^\alpha S_a(t-t_{j+1}) w(t_{j+1})\n_{L^q(I_j,L^r)}&\leq\n D^\alpha S_a(t-t_j)w(t_j)\n_{L^q(I_j,L^r)}+c\varepsilon+c\eta^{\frac{4}{d-2}}B_j+cB_j^{\frac{d+2}{d-2}}\\
&\leq\gamma_j+c\eta^{\frac{4}{d-2}}B_j+2\gamma_j\leq3\gamma_j+c\eta^{\frac{4}{d-2}}4\gamma_j
\end{align*}
and choosing $\eta>0$ small we get $\gamma_{j+1}\leq 5\gamma_j$.
 Note that by \eqref{fri}
\begin{align*}
\n D^\alpha e^{-i(t-t_j)\la}w(t_j)\n_{L^q(I,L^r)}&\lesssim\n e^{-i(t-t_j)\la}w(t_j)\n_{S(I)}^{1-\alpha}\n\nabla e^{-i(t-t_j)\la}w(t_j)\n_{W(I)}^\alpha\\
&\lesssim\n e^{-i(t-t_j)\la}w(t_j)\n_{S(I)}^{1-\alpha}\n w(t_j)\n_{\dot{H}^1}^\alpha.
\end{align*}
Therefore by the hypothesis
that $\gamma_0 \leq \varepsilon^{1-\beta}A'+c\varepsilon.$ Iterating, we have $\gamma_j \leq  5^j(\varepsilon^{1-\beta}A'+c \varepsilon)$ if $B_j \leq c_0.$ 
Thus $B_j\leq4\gamma_j\leq 5^j4(\varepsilon^{1-\beta}A'+c \varepsilon)$ if $B_j \leq c_0.$ Choose $\varepsilon_0=\varepsilon_0(c,l)=\varepsilon_0(c,M,M',\eta)=\varepsilon_0(c,M,A,d)>0$ so that $5^l4(\varepsilon_0^{1-\beta}A'+c \varepsilon_0)=c_0$.

Therefore for $0<\varepsilon<\varepsilon_0$ we have $\n D^\alpha w\n_{L^q(I,L^r)}\leq 5^ll4(\varepsilon^{1-\beta}A'+c \varepsilon)$ and hence by \eqref{fhi} $\n w\n_{S(I)}\leq c5^ll4(\varepsilon^{1-\beta}A'+c \varepsilon)$. Using Strichartz again we get $\n w(t)\n_{\dot{H}^1}\leq C(\varepsilon^{1-\beta}A'+c \varepsilon)$ for all $t\in I$. This proves the required estimates and hence the theorem.
\end{proof}

\begin{rmk}\label{rm3}
{\rm
Since  we are applying the Strichartz estimates on the $\alpha$-fractional derivative, the equivalences of the homogeneous Sobolev norms $\n\cdot\n_{\dot{W}_a^{\alpha,r}}$, $\n\cdot\n_{\dot{W}^{\alpha,r}}$ and $\n\cdot\n_{\dot{W}_a^{\alpha,\tilde{r}'}}$, $\n\cdot\n_{\dot{W}^{\alpha,\tilde{r}'}}$ play  roles here. Note that the first equivalence is a consequence of the restriction on $a$ and the second one is true for all $a$. In the case when $a=0$ such issue does not arise. 
}\end{rmk}

\subsection{Scattering of Solutions}\label{scat}
In this subsection, we outline a  proof of the scattering result, 
see Theorem \ref{mt} below for the exact statement. First, we define the ground state solution $W_a$ and energy of a solution of \eqref{61}:

\begin{dfn}\hspace{0cm} 
{\rm(i)} Given $a<\left(\frac{d-2}{2}\right)^2$, we define  $\beta >0 $ via  $a=(\frac{d-2}{2})^2 [1-\beta^2] $. 
Then define the function (ground state solution) by
$W_a(x):=[d(d-2)\beta^2]^{\frac{d-2}{4}} \big[  \frac{|x|^{\beta -1}}{1+ |x|^{2\beta}}\big]^{(d-2)/{2}}.$ \\
{\rm(ii)} By 
$E_a(u(t))= \int_{\mathbb R^d}\big(\frac{1}{2} |\nabla u(t,x)|^2- \frac{a}{2|x|^2}|u(t,x)|^2 - \frac{1}{2^*} |u(t,x)|^{2^*}\big) dx$, we define the Energy $E_a(u)$ of a solution $u$ corresponding to our problem.
\end{dfn}
For details of ground state solutions, one can see \cite{terracini1996positive,catrina2001caffarelli,dolbeault2016rigidity}.
Note that the energy $E(u)$ is conserved for a solution $u$ to \eqref{61} throughout the maximal interval of existence, see \cite[Lemma 3.6]{okazawa2012cauchy}.
Now we are in a position to state the scattering result: 
\begin{theorem}[Scattering of Solution]\label{mt}
 Let $d\in\{3,4,5\}$ and $a<\left(\frac{d-2}{2}\right)^2-\left(\frac{d-2}{d+2}\right)^2.$ Assume that  $E_{a}(u_0) < E_{a\vee 0}(W_{a\vee 0})$ and $\|u_0\|_{\dot{H}^1_a} < \|W_{a\vee 0}\|_{\dot{H}^1_{a\vee0}}$ and $u_0$ is radial. Then the solution $u$ to \eqref{61}
with data at $t=0$  is defined for all time with $\n u\n_{S(\R)}<\infty$ and there exists $u_{0, +}, u_{0, -}$ in $\dot{H}^1$ such that 
$$ \lim_{t\to + \infty} \|u(t)- e^{it \mathcal{L}_a}u_{0, +} \|_{\dot{H}^1}=0,\quad \lim_{t\to - \infty} \|u(t)- e^{it \mathcal{L}_a}u_{0, -} \|_{\dot{H}^1}=0.$$

\end{theorem}  

Before giving the proof of Theorem \ref{mt}, we state a few preliminaries form early works:

\begin{theorem}[Coercivity, see Corollary 7.6 in \cite{killip2017energy}] \label{ct} Let $d\geq 3$ and $a<\left( \frac{d-2}{2}\right)^{2}.$ Let  $u:I \times \mathbb R^d \to \mathbb C$ be a solution to $\eqref{61}$ with initial data  $u(t_ 0 ) = u_ 0\in \dot{H}^1(\mathbb R^d)$ for some $t_0 \in I.$ 
Assume $E_a (u_ 0 ) \leq (1 -\delta_0 )E_{a\vee 0} (W_{a\vee 0} )$ for some $\delta_0 >0$. Then there exist positive constants  $\delta_1$ and $c$ depending on $d, a, \delta_0$, such that 
 if $\|u_0\|_{\dot{H}^1_a} \leq \|W_{a\vee 0} \|_{\dot{H}^1_{a\vee 0}},$ then for all $t\in I$
\begin{enumerate}
\item[(i)] \label{uec} $ \|u(t)\|_{\dot{H}^1_a} \leq (1-\delta_1) \|W_{a\vee 0} \|_{\dot{H}^1_{a\vee 0}}$.
\item[(ii)] $\int_{\mathbb R^d} |\nabla u(t,x)|^2 + \frac{a}{|x|^2} |u(t,x)|^2- |u(t,x)|^{\frac{2d}{d-2}} dx \geq c \|u(t)\|_{\dot{H}^1_a}^2.$
\item[(iii)] $ c\|u(t)\|^2_{\dot{H}^1_a} \leq 2 E_a(u) \leq  \|u(t)\|^2_{\dot{H}^1_a}$.
\end{enumerate}
\end{theorem}

\begin{theorem}[Concentration  compactness, see Lemma 4.3 in \cite{kenig2006global}, Theorem 3.1 in \cite{killip2017energy}, \cite{yang2020scattering}]\label{cc}
Assume $a<\left(\frac{d-2}{2}\right)^2-\left(\frac{d-2}{d+2}\right)^2.$ Let  $\{v_{0,n} \}\in \dot{H}^{1}(\rd)$, $\|v_{0,n}\|_{\dot{H}^1}< A,$ $v_{0,n}$ is radial for all $n\in \mathbb N.$ Assume that $\|e^{it\mathcal{L}_a}v_{0,n}\|_{S(\mathbb R)} \geq \delta >0,$ where $\delta= \delta(A)$ is as in Proposition \ref{gsd}. Then there exist a sequence $\{V_{0, j}\}_{j=1}^{\infty}$ in $\dot{H}^1(\rd)$, a subsequence of $\{v_{0,n}\}$ (which we still call $\{v_{0,n}\}$) and a couple $(\lambda_{j,n}, t_{j,n}) \in(0,\infty) \times \mathbb R,$ with 
$$ \frac{\lambda_{j,n}}{\lambda_{j',n}} + \frac{\lambda_{j',n}}{\lambda_{j,n}} +  \frac{|t_{j,n}- t_{j', n}|}{\lambda^2_{j',n}} \to \infty $$
as $n\to \infty$ for $j\neq j'$ such that 
$
 \|V_{0,1}\|_{\dot{H}^1} \geq \alpha_{0}(A) >0.
$
If $V_{j}^{l} (x,t):= e^{it \mathcal{L}_a} V_{0,j}(x),$ then, given $\epsilon_0 >0,$
there exists  $J=J(\epsilon_0)$ and  $\{w_n\}_{n=1}^{\infty} \in \dot{H}^{1}(\rd),$ so that 
\begin{enumerate}
\item[(i)] $v_{0,n}= \sum_{j=1}^{J} \frac{1}{\lambda_{j,n}^{(d-2)/2}} V_{j}^{l} \left(-\frac{t_{j,n}}{\lambda_{j,n}^2},\frac{x}{\lambda_{j,n}}\right) +w_n$
\item[(ii)] $\n e^{it\mathcal{L}_a} w_n\|_{S(\mathbb R)} \leq \epsilon_0$
\item [(iii)]$\|v_{0,n}\|^2_{\dot{H}^1_a} = \sum_{j=1}^{J} \|V_{0,j}\|^2_{\dot{H}^1_a} + \|w_n\|^2_{\dot{H}^1_a} +o(1)$ as $n\to\infty$
\item[(iv)] $E_a(v_{0,n})= \sum_{j=1}^{J} E_a\left(V_{j}^{l} \left(\frac{-t_{j,n}}{\lambda_{j,n}^2}\right)\right) +E_a(w_n) + o(1)$ as $n\to\infty$.
\end{enumerate}
In addition we may assume that for each $j$ either $\frac{t_{j,n}}{\lambda_{j,n}^2}\equiv0$ or $\frac{t_{j,n}}{\lambda_{j,n}^2}\to\infty$ as $n\to\infty$.
\end{theorem}
\begin{rmk}\label{rm2}
{\rm The original result \cite[Theorem 3.1]{killip2017energy} says we would get a sequence $\{x_{j,n}\}$ along with $\{\lambda_{j,n}\},\{t_{j,n}\}$. But due to the radial situation we can take $x_{j,n}=0$ for all $j,n$'s. }
\end{rmk}
\begin{prp}[Localized virial identity] \label{vi} Let $\phi \in C_0^{\infty}(\mathbb R^d), t \in [0, T_{+}(u_0)).$ Then for $u$ satisfying $i\partial_tu+\Delta u-Vu+|u|^{4/(d-2)}u=0$ we have
\begin{enumerate}
\item[(i)]$\frac{d}{dt} \int_{\rd} |u|^2 \phi  = 2 \text{Im} \int_{\rd} \bar{u} \nabla u\cdot \nabla \phi dx$
\item[(ii)] $\frac{d^2}{dt^2} \int_{\rd} |u|^2\phi  =4\sum_{i,j}\text{Re}\int_{\rd}\partial_{x_ix_j}\phi \partial_{x_i}u\partial_{x_j}\bar{u}-\int_{\rd}[ \Delta^2\phi+2\nabla\phi\cdot\nabla V]|u|^2-\frac{4}{d}\int_{\rd}\Delta\phi|u|^{2^*}.$
\end{enumerate}
\end{prp}
\begin{proof}[{\bf Proof}]
See \cite[Lemma 7.2]{killip2013nonlinear} by Killip and Visan.
\end{proof}

Now let us give a shorthand notation to an $u_0\in\hd$ for which scattering happens:
\begin{dfn}
Let $u_0\in \dot{H}_a^1(\rd)$ with $\|u_0\|_{\dot{H}^1_a} < \|W_{a\vee0}\|_{\dot{H}^{1}_{a\vee0}}$ and $E_a(u_0)< E_{a\vee 0}(W_{a\vee 0})$. We say that $(SC)(u_0)$ holds, if the maximal interval $I$ of existence of the solution $u$ to \eqref{61} with initial data $u_0$ at $t_0$, is $\R$ and $\n u\n_{S(\R)}<\infty$.
\end{dfn}


Note that, because of Proposition \ref{gsd}, Strichartz and Sobolev inequality, if  $\|u_0\|_{\dot{H}^1_a} \leq \delta,$
 $(SC)(u_0)$ holds.
Thus, in light of Theorem \ref{ct}, there exists  $\eta_0 >0$   such that $\|u_0\|_{\dot{H}_a^1} < \|W_{a \vee 0}\|_{ \dot{H}^1_{a\vee0}}$, 
   $E_a(u_0)< \eta_0$,  then $(SC)(u_ 0)$ holds. 
Thus, there exists a number $E_C$,
with  $0<\eta_0 \leq E_C \leq E_{a\vee 0 }(W_{a\vee 0})$, such that, if $ \|u_0\|_{\dot{H}_a^1} < \|W_{a \vee 0}\|_{ \dot{H}^1_{a\vee0}}$ 
and  $E_a(u_ 0) < E_ C$, then $(SC)(u_ 0)$ holds and $E_C$ is optimal with this property.
Note that \begin{align*}
E_C=\sup\left\{E \in (0, E_{a \vee 0} (W_{a \vee 0})):\  \|u_0\|_{\dot{H}_a^1} < \|W_{a \vee 0}\|_{ \dot{H}^1_{a\vee0}} , E_a(u_0) < E  \ \Rightarrow  (SC)(u_0) \text{ holds}\right\}
\end{align*}
and  $E_C\leq E_{a\vee 0} (W_{a \vee 0})$. Assuming  $E_ C < E_{a\vee 0}(W_{a \vee 0} )$, we have existence of a critical solution with some compactness property, namely we have the following result:

\begin{prp}\label{ec}  Let $E_ C < E_{a\vee 0}(W_{a \vee 0} )$. Then there exists   $u_{0, C} \in \dot{H}^1(\rd)$ with
$$ E_a(u_{0, C})= E_{C} < E_{a\vee 0}(W_{a\vee 0}),\quad \|u_{0, C} \|_{\dot{H}^1_a} < \|W_{a\vee 0}\|_{\dot{H}^1_{a\vee0}}$$
such that,  if $u_C$ is the solution of  $(NLS)_a$ with initial data $u_{ 0,C}$ at $t=0$ and maximal interval of existence $I$, then $\|u_{C}\|_{S(I)}= \infty$. In addition $u_C$ has the following property:  If $\|u_{C}\|_{S(I_+)}= \infty$ then there exists a function $\lambda:I_+\to(0,\infty)$ such that the set 
$$K= \left\{ v(t,x): v(t,x) = \frac{1}{\lambda(t)^{(d-2)/2}}u_{C} \left(t,\frac{x}{\lambda(t)}\right) \right\}$$
has compact closure  in $\dot{H}^{1}(\rd)$. A corresponding conclusion is
reached if  $\|u_{C}\|_{S(I_-)}= \infty,$ where   $I_+=(0,\infty)\cap I,I_-= (-\infty, 0) \cap I.$
\end{prp}
\begin{proof}[{\bf Proof}]
The existence of $u_C$ follows exactly in the same way as in \cite[Proposition 4.1]{kenig2006global} once we have Theorems \ref{ltp} and \ref{cc}. For the existence of $\lambda$ we go in the way of proof of \cite[Proposition 4.2]{kenig2006global} along with Theorem \ref{63} (with $G$ defined as Lemma \ref{16} and $H=\dot{H}^1(\rd)$).
\end{proof}

Now we have the following rigidity result:
\begin{prp}\label{62}
Let $u_0\in\hd$ such that
$ E_a(u_0) < E_{a\vee 0}(W_{a\vee 0}),\  \|u_0 \|_{\dot{H}^1_a} < \|W_{a\vee 0}\|_{\dot{H}^1_{a\vee0}}$ and  $u$ be the solution to \eqref{61} with $u(0,\cdot)=u_0$. Assume there exists a function $\lambda:I_+\to(0,\infty)$ such that the set 
$$K= \left\{ v(t,x): v(t,x) = \frac{1}{\lambda(t)^{(d-2)/2}}u\left(t,\frac{x}{\lambda(t)}\right) \right\}$$
has compact closure  in $\dot{H}^{1}(\rd)$. Then $u=0$.
\end{prp}
\begin{proof}[{\bf Proof}]
The proof is similar to that of  \cite[Proposition 5.3]{kenig2006global} once we have Theorem \ref{ct} and Proposition \ref{vi}.
\end{proof}
\begin{proof}[{\bf Proof of Theorem \ref{mt}}]
Note that Theorem \ref{mt} is the assertion $E_C= E_{a\vee 0} (W_{a \vee 0})$. If not assume $E_C< E_{a\vee 0} (W_{a \vee 0})$. By Proposition \ref{ec} we have existence of a minimal solution $u_C$ satisfying the assumption of Proposition \ref{62}.
Applying Proposition \ref{62} to $u_C$ we conclude that $u_C=0$ which is a contradiction as we had $\n u_C\n_{S(I)}=\infty$ from Proposition \ref{ec}. 
 \end{proof}
\begin{rmk}
{\rm The non-radial data can also be dealt in this technique provided one can bound the sequence $\{x_{1,n}\}$ in concentration compactness result, see Remark \ref{rm2} and Theorem \ref{cc}. In fact this is proved in \cite{yang2020scattering} for dimension $d=4,5$.
However, the non-radial case in dimension $d=3$ is still open.}
\end{rmk}

\subsection*{Acknowledgement}I am thankful to Sandeep K. and D. Bhimani for various help and suggestions. 
I thank M. Milman for helping me to find the reference \cite{janson1988interpolation}.
  I am also thankful to the unknown referee for valuable comments and suggestions.
\bibliographystyle{siam}
\bibliography{str}

\begin{thebibliography}{10}

\bibitem{bergh2012interpolation}
{\sc J.~Bergh and J.~L{\"o}fstr{\"o}m}, {\em Interpolation spaces: {A}n
  introduction}, vol.~223, Springer Science \& Business Media, 2012.

\bibitem{bouclet2018uniform}
{\sc J.-M. Bouclet and H.~Mizutani}, {\em Uniform resolvent and {S}trichartz
  estimates for {S}chr{\"o}dinger equations with critical singularities},
  Transactions of the American Mathematical Society, 370 (2018),
  pp.~7293--7333.

\bibitem{burq2003strichartz}
{\sc N.~Burq, F.~Planchon, J.~G. Stalker, and A.~S. Tahvildar-Zadeh}, {\em
  Strichartz estimates for the wave and {S}chr{\"o}dinger equations with the
  inverse-square potential}, Journal of functional analysis, 203 (2003),
  pp.~519--549.

\bibitem{catrina2001caffarelli}
{\sc F.~Catrina and Z.-Q. Wang}, {\em On the {C}affarelli-{K}ohn-{N}irenberg
  inequalities: {S}harp constants, existence (and nonexistence), and symmetry
  of extremal functions}, Communications on Pure and Applied Mathematics: A
  Journal Issued by the Courant Institute of Mathematical Sciences, 54 (2001),
  pp.~229--258.

\bibitem{cazenave2003semilinear}
{\sc T.~Cazenave}, {\em Semilinear {S}chr{\"o}dinger Equations}, vol.~10,
  American Mathematical Soc., 2003.

\bibitem{cazenave1988cauchy}
{\sc T.~Cazenave and F.~B. Weissler}, {\em The {C}auchy problem for the
  nonlinear {S}chr{\"o}dinger equation in ${H}^1$}, manuscripta mathematica, 61
  (1988), pp.~477--494.

\bibitem{cazenave1992rapidly}
\leavevmode\vrule height 2pt depth -1.6pt width 23pt, {\em Rapidly decaying
  solutions of the nonlinear {S}chr{\"o}dinger equation}, Communications in
  mathematical physics, 147 (1992), pp.~75--100.

\bibitem{dolbeault2016rigidity}
{\sc J.~Dolbeault, M.~J. Esteban, and M.~Loss}, {\em Rigidity versus symmetry
  breaking via nonlinear flows on cylinders and {E}uclidean spaces},
  Inventiones mathematicae, 206 (2016), pp.~397--440.

\bibitem{foschi2005inhomogeneous}
{\sc D.~Foschi}, {\em Inhomogeneous {S}trichartz estimates}, Journal of
  Hyperbolic Differential Equations, 2 (2005), pp.~1--24.

\bibitem{fujiwara1979construction}
{\sc D.~Fujiwara}, {\em A construction of the fundamental solution for the
  {S}chr{\"o}dinger equation}, Journal d’Analyse Math{\'e}matique, 35 (1979),
  pp.~41--96.

\bibitem{ginibre1985global}
{\sc J.~Ginibre and G.~Velo}, {\em The global {C}auchy problem for the non
  linear {S}chr{\"o}dinger equation revisited}, in Annales de l'Institut Henri
  Poincare (C) Non Linear Analysis, vol.~2, Elsevier, 1985, pp.~309--327.

\bibitem{huang2019function}
{\sc L.~Huang and D.~Yang}, {\em On function spaces with mixed norms---{A}
  survey}, arXiv preprint arXiv:1908.03291,  (2019).

\bibitem{janson1988interpolation}
{\sc S.~Janson}, {\em On interpolation of multi-linear operators}, in Function
  spaces and applications, Springer, 1988, pp.~290--302.

\bibitem{kato1994q}
{\sc T.~Kato}, {\em An $ {L}^{q,r} $-theory for nonlinear {S}chr{\"o}dinger
  equations}, in Spectral and scattering theory and applications, Mathematical
  Society of Japan, 1994, pp.~223--238.

\bibitem{kato2013perturbation}
\leavevmode\vrule height 2pt depth -1.6pt width 23pt, {\em Perturbation theory
  for linear operators}, vol.~132, Springer Science \& Business Media, 2013.

\bibitem{keel1998endpoint}
{\sc M.~Keel and T.~Tao}, {\em Endpoint {S}trichartz estimates}, American
  Journal of Mathematics, 120 (1998), pp.~955--980.

\bibitem{kenig}
{\sc C.~E. Kenig}, {\em Global well-posedness, scattering and blow up for the
  energy-critical, focusing, non-linear {S}chr{\"o}dinger and wave equations},
  Lecture notes.

\bibitem{kenig2006global}
{\sc C.~E. Kenig and F.~Merle}, {\em Global well-posedness, scattering and
  blow-up for the energy-critical, focusing, non-linear {S}chr{\"o}dinger
  equation in the radial case}, Inventiones mathematicae, 166 (2006),
  pp.~645--675.

\bibitem{killip2017energy}
{\sc R.~Killip, C.~Miao, M.~Visan, J.~Zhang, and J.~Zheng}, {\em The
  energy-critical {NLS} with inverse-square potential}, Discrete \& Continuous
  Dynamical Systems-A, 37 (2017), pp.~3831--3866.

\bibitem{killip2018sobolev}
{\sc R.~Killip, C.~Miao, M.~Visan, J.~Zhang, and J.~Zheng}, {\em Sobolev spaces
  adapted to the {S}chr{\"o}dinger operator with inverse-square potential},
  Mathematische Zeitschrift, 288 (2018), pp.~1273--1298.

\bibitem{killip2013nonlinear}
{\sc R.~Killip and M.~Visan}, {\em Nonlinear {S}chr{\"o}dinger equations at
  critical regularity}, Evolution equations, 17 (2013), pp.~325--437.

\bibitem{koh2011improved}
{\sc Y.~Koh}, {\em Improved inhomogeneous {S}trichartz estimates for the
  {S}chr{\"o}dinger equation}, Journal of mathematical analysis and
  applications, 373 (2011), pp.~147--160.

\bibitem{linares2014introduction}
{\sc F.~Linares and G.~Ponce}, {\em Introduction to nonlinear dispersive
  equations}, Springer, 2014.

\bibitem{mizutani2020uniform}
{\sc H.~Mizutani, J.~Zhang, and J.~Zheng}, {\em Uniform resolvent estimates for
  {S}chr{\"o}dinger operator with an inverse-square potential}, Journal of
  Functional Analysis, 278 (2020), p.~108350.

\bibitem{oh1988existence}
{\sc Y.-G. Oh}, {\em Existence of semiclassical bound states of nonlinear
  {S}chr{\"o}dinger equations with potentials of the class (v) a},
  Communications in Partial Differential Equations, 13 (1988), pp.~1499--1519.

\bibitem{okazawa2012cauchy}
{\sc N.~Okazawa, T.~Suzuki, and T.~Yokota}, {\em Cauchy problem for nonlinear
  {S}chr{\"o}dinger equations with inverse-square potentials}, Applicable
  Analysis, 91 (2012), pp.~1605--1629.

\bibitem{o1963convolution}
{\sc R.~O’Neil}, {\em Convolution operators and ${L}^{p, q} $ spaces}, Duke
  Mathematical Journal, 30 (1963), pp.~129--142.

\bibitem{pierfelice2006strichartz}
{\sc V.~Pierfelice}, {\em Strichartz estimates for the {S}chr{\"o}dinger and
  heat equations perturbed with singular and time dependent potentials},
  Asymptotic Analysis, 47 (2006), pp.~1--18.

\bibitem{schonbek1979decay}
{\sc T.~Schonbek}, {\em Decay of solutions of {S}chr{\"o}dinger equations},
  Duke Mathematical Journal, 46 (1979), pp.~203--213.

\bibitem{strichartz1977restrictions}
{\sc R.~S. Strichartz}, {\em Restrictions of {F}ourier transforms to quadratic
  surfaces and decay of solutions of wave equations}, Duke Mathematical
  Journal, 44 (1977), pp.~705--714.

\bibitem{terracini1996positive}
{\sc S.~Terracini}, {\em On positive entire solutions to a class of equations
  with a singular coefficient and critical exponent}, Advances in Differential
  Equations, 1 (1996), pp.~241--264.

\bibitem{vilela2007inhomogeneous}
{\sc M.~Vilela}, {\em Inhomogeneous {S}trichartz estimates for the
  {S}chr{\"o}dinger equation}, Transactions of the American Mathematical
  Society, 359 (2007), pp.~2123--2136.

\bibitem{weinstein1985symbol}
{\sc A.~Weinstein}, {\em A symbol class for some {S}chr{\"o}dinger equations on
  $\mathbb{R}^n$}, American Journal of Mathematics,  (1985), pp.~1--21.

\bibitem{yajima1987existence}
{\sc K.~Yajima}, {\em Existence of solutions for {S}chr{\"o}dinger evolution
  equations}, Communications in Mathematical Physics, 110 (1987), pp.~415--426.

\bibitem{yang2020scattering}
{\sc K.~Yang}, {\em Scattering of the energy-critical {NLS} with inverse square
  potential}, Journal of Mathematical Analysis and Applications,  (2020),
  p.~124006.

\bibitem{yang2020scattering0}
\leavevmode\vrule height 2pt depth -1.6pt width 23pt, {\em Scattering of the
  focusing energy-critical {NLS} with inverse square potential in the radial
  case}, Communications on Pure \& Applied Analysis,  (2020), p.~1.

\bibitem{zelditch1983reconstruction}
{\sc S.~Zelditch}, {\em Reconstruction of singularities for solutions of
  {S}chr{\"o}dinger's equation}, Communications in Mathematical Physics, 90
  (1983), pp.~1--26.

\end{thebibliography}
\end{document}